\documentclass[11pt]{amsart}
\usepackage{amsfonts,amssymb,amsmath,amsthm}
\usepackage{mathtools}
\usepackage[T1]{fontenc}
\usepackage[utf8]{inputenc}
\usepackage{enumerate}
\usepackage[colorlinks=true,citecolor=cyan,urlcolor=blue,linkcolor=blue]{hyperref}
\usepackage[margin=2.9cm]{geometry}

\newtheorem{theorem}{Theorem}[section]
\newtheorem{proposition}[theorem]{Proposition}
\newtheorem{corollary}[theorem]{Corollary}
\newtheorem{lemma}[theorem]{Lemma}
\theoremstyle{definition}
\newtheorem{definition}[theorem]{Definition}
\newtheorem{remark}[theorem]{Remark}

\def\R{\mathbb R}
\def\C{\mathbb C}
\def\H{\mathbb H}
\def\Z{\mathbb Z}
\def\Q{\mathbb Q}

\def\P{\mathbb P}
\def\M{\mathrm M}
\def\S2{\mathbb S^2}
\newcommand{\SL}{\mathrm{SL}}
\newcommand{\PSL}{\mathrm{PSL}}
\newcommand{\GL}{\mathrm{GL}}
\newcommand{\SU}{\mathrm{SU}}
\newcommand{\Sp}{\mathrm{Sp}}
\newcommand{\SO}{\mathrm{SO}}
\newcommand{\Uni}{\mathrm{U}}
\newcommand{\Isom}{\mathrm{Isom}}
\newcommand{\Hom}{\mathrm{Hom}}
\newcommand{\Tr}{\mathrm{Tr}}
\def\tr{\mathrm{tr}}
\def\Re{\mathrm{Re}}
\def\Imm{\mathrm{Im}}
\def\hH{\widehat{\H}}
\def\X{\mathbb X}
\newcommand{\ddet}{\mathrm{det}}
\newcommand{\Fix}{\mathrm{Fix}}
\newcommand{\Cent}{\mathrm{Z}}

\newcommand{\Mreg}{\mathcal M^{\mathrm{reg}}}
\newcommand{\MFreg}{\mathcal M_{F}^{\mathrm{reg}}}
\newcommand{\MFregstar}{\mathcal M_{F,*}^{\mathrm{reg}}}

\begin{document}
\title[Trace Coordinates and Local Fenchel--Nielsen Parameters ]{Trace Coordinates and Local Fenchel--Nielsen Parameters in Real Hyperbolic \(5\)-Space}

\author[K. Gongopadhyay, S. Kalane \and A. Mukherjee]
{Krishnendu Gongopadhyay, Sagar B.  Kalane \and Abhishek Mukherjee}

\address{Indian Institute of Science Education and Research (IISER) Mohali,
Knowledge City, Sector 81, S.A.S.\ Nagar 140306, Punjab, India}
\email{krishnendu@iisermohali.ac.in}

\address{Dr. Sagar Kalane Academy, 41 City Hub, Gadital, Hadapsar, Pune  411028, India}
\email{sagark327@gmail.com}

\address{Department of Mathematics, Kalna College, Kalna, Dist.\ Burdwan, West Bengal 713409, India}
\email{abhiquaternion@gmail.com, abhishekmukherjee@kalnacollege.ac.in}


\date{\today}
\begin{abstract} 
We study representations of fundamental groups of closed orientable surfaces into
\(\SL(2,\mathbb H)\), whose projectivization \(\PSL(2,\mathbb H)\) is the
group of orientation-preserving isometries of real hyperbolic \(5\)-space,
with the aim of developing a quaternionic analogue of Fenchel--Nielsen theory.
We give a normal form for pairs of regular loxodromic elements with disjoint
fixed point sets and show that on an open dense generic locus, the
corresponding moduli space is \(15\) dimensional. We also associate fifteen
word traces whose differentials are linearly independent on an open dense 
subset,  and hence give local real-analytic coordinates there.

For a pair of pants with prescribed regular loxodromic boundary conjugacy
classes, we show that the relative deformation space is locally
\(6\) dimensional. Combining this internal pants data with the
three-dimensional boundary conjugacy data and the three-dimensional
centralizer gluing freedom gives a local parameter decomposition for closed
surface group representations. For a closed surface of genus \(g\), this gives
\(30g-30\) real parameters. 
\end{abstract}
\maketitle


\section{Introduction}\label{sec-intro}

Let $S$ be a closed orientable surface of genus $g\ge 2$.  The Teichm\"uller space
$\mathcal T(S)$ may be described algebraically as the space of discrete and faithful
representations
\[
\rho\colon\pi_1(S)\longrightarrow \PSL(2,\R)
\]
up to conjugation. Equivalently it is the deformation space of marked hyperbolic
structures on $S$.  Fix a pant decomposition of $S$, that is, a maximal collection
$\mathcal C=\{\gamma_1,\dots,\gamma_{3g-3}\}$ of pairwise disjoint, homotopically distinct,
essential simple closed curves.  Cutting $S$ along $\mathcal C$ produces $2g-2$ pairs of pants.
Recording a length parameter for each boundary curve and a twist parameter for each gluing,
Fenchel and Nielsen showed that $\mathcal T(S)$ is parametrised by
$(\R_{>0})^{3g-3}\times\R^{3g-3}$.

From the algebraic point of view a pair of pants $Y$ has free fundamental group of rank
two, and a representation $\rho\colon\pi_1(Y)\to\PSL(2,\R)$ is determined by three
boundary elements $A,B,C$ with $ABC=I$.  In the classical case the work of Fricke and
Vogt shows that the conjugacy class of such a triple is determined by trace data.  When
two pairs of pants are glued one must in addition understand the centralizer of the
boundary element along which the gluing is performed; this produces the twist parameter.
Fenchel--Nielsen coordinates therefore rest on two ingredients: invariants of loxodromic
boundary elements, and invariants describing the relative position of two pairs of pants
under gluing.

This philosophy has been extended to several higher rank and higher
dimensional settings: quasi-Fuchsian representations into $\SL(2,\C)$
\cite{kour,tan,Mas}, convex projective structures and Hitchin-type
coordinates for $\SL(3,\R)$ \cite{gold1,gold2}, complex hyperbolic
analogues for $\SU(2,1)$ \cite{parkplat,park1,wil}, and quaternionic
hyperbolic Fenchel--Nielsen coordinates in \(\Sp(2,1)\) \cite{kgsk} and local coordinates for quaternionic (and complex) 
hyperbolic pairs in $\Sp(n, 1)$ \cite{kgsk2}. More recently,
Fenchel--Nielsen coordinates for $\SL(3,\C)$ representations were
developed by D\'avila Figueroa and Parker \cite{fp}. In each case one seeks algebraic
invariants playing the role of length and twist coordinates.

Let $\H$ denote the division ring of Hamilton quaternions. 
The group $\PSL(2,\H)$, acting by quaternionic M\"obius transformations on
$\widehat{\H}\cong \mathbb{S}^4$, is naturally identified with the group of
orientation-preserving isometries of real hyperbolic $5$-space, see, for example, \cite{kg}, \cite{ps}. For quaternionic Kleinian groups and arithmetic examples in this
setting, see also \cite{DVV}. 

In the final section of \cite{kgsk}, the groups \(\GL(2,\H)\) and \({\rm PGL}(2,\H)\) were considered using spatial parameters arising from the eigenspaces corresponding to eigenvalue classes, which were termed ``projective points.'' These parameters were used to obtain parameter counts for generic loxodromic pairs and geometric surface-group representations; see in particular \cite[Sec.~7.4]{kgsk}. Since \({\rm PGL}(2,\H)\cong\PSL(2,\H)\), every projective class can be represented by an element of \(\SL(2,\H)\), uniquely up to the central subgroup \(\{\pm I\}\), by normalizing the Dieudonn'e determinant. This provides the starting point for the present work. We work in the linear group \(\SL(2,\H)\), where characteristic polynomials and trace-type invariants are available, while the associated projective deformation theory reflects the geometry of real hyperbolic \(5\)-space. Our aim is to develop this picture further by establishing the corresponding local analytic and deformation-theoretic structure underlying a Fenchel--Nielsen theory in this setting.

A first indication that the quaternionic theory is substantially different
from the classical one already appears at the level of a single pair of
pants. For $\SL(2,\R)$ and $\SL(2,\C)$, fixing the three boundary
conjugacy classes determines a pair of pants representation locally up to
conjugacy. This rigidity fails dramatically for $\SL(2,\H)$: on the
generic locus, a pair of pants with three prescribed regular loxodromic
boundary conjugacy classes has a six-dimensional relative deformation
space. For comparison, the corresponding relative dimension for
$\SU(2,1)$ is two. Thus, in the quaternionic setting, lengths and twists
alone cannot form a complete Fenchel--Nielsen coordinate system: six
additional parameters are required for every pair of pants. We call this `internal pants data'. 
This six-dimensional internal moduli space is the main structural feature
that guides the paper.   To understand it, and ultimately to obtain local
coordinates for surface-group representations, one must first understand
three related pieces of geometric structures: the conjugacy data of an individual
loxodromic element, the simultaneous conjugacy class of a loxodromic pair,
and the centralizer of a boundary element governing the gluing of adjacent
pairs of pants. We explicitly develop each of these to serve our main aim.

The quaternionic setting is genuinely subtler than the real and complex ones.
Because $\mathbb H$ is non-commutative, one must distinguish left from right
eigenvalues, and right eigenvalues are defined only up to similarity. Thus an
eigenvalue similarity class remembers its modulus and argument but forgets an family of possible quaternionic axes that can be identified with $\mathbb S^2$. These additional directions are
ultimately responsible for much of the extra deformation freedom that has no
analogue over $\mathbb R$ or $\mathbb C$. In \cite{kgsk}, the same freedom was
encoded by \emph{projective points}: a projective point is a point of the
corresponding eigensphere $\mathbb{CP}^1\cong \mathbb{S}^2$, equivalently,  a choice of
eigenvalue representative within its similarity class. In the present paper
we describe this parameter directly by the associated quaternionic axis
$I\in \mathbb{S}^2$, and refer to it as an \emph{axial direction}.

We make these ideas precise first for individual loxodromic elements.
The conjugacy theory of quaternionic M\"obius transformations has been studied
extensively;  see, for example, \cite{cpw,ps,kg,foreman,gkl,wi}, also see \cite{wat}. Here we refine this  theory in a form suited to understand the deformation space in our set up. We determine exactly which
triples of three real spectral invariants arise from regular loxodromic elements (those with two non-real eigenvalue
similarity classes), 
and identify the space of their  conjugacy classes real analytically with an explicit open region
\(\mathcal D\subset\R^3\). Although one of these invariants, $\nu(A)$, is not itself a trace,
Remark~\ref{rem:nu-resolvent} shows that it is recovered as the unique
root greater than $2$ of a resolvent cubic determined by
$\tr A$, $\tr A^{-1}$, and $\tr(A^2)$. Thus the conjugacy classification
can be expressed entirely in terms of trace data.

We also compute the centralizer of a regular loxodromic element in
\(\SL(2,\H)\), showing that it is isomorphic to
\(\R_{>0}\times\Uni(1)\times\Uni(1)\); compare the corresponding
\(\GL(2,\H)\) computation in \cite{kg}. These three dimensions correspond to the twist parameters associated with each gluing curve. Geometrically, they consist of one translational parameter along the axis and two independent rotational parameters about it.

We next study regular loxodromic pairs with disjoint fixed-point sets. We give an
explicit normal form separating the spectral data, the quaternionic cross ratio,
and the residual axial data, and obtain a real-analytic description of an open
dense \(15\)-dimensional moduli space. We further introduce framings of
loxodromic pairs and the corresponding framed moduli space. The
axial directions then appear naturally as fiber coordinates of the forgetful map
to the unframed moduli space. This provides a precise geometric interpretation of  the 
projective-point parameters used in \cite{kgsk}.

We also obtain explicit local trace coordinates for pairs of regular
loxodromic elements in \(\SL(2,\H)\). Quaternionic trace invariants under
unitary similarity were studied by \v{D}okovi\'c and Smith \cite{DS}, who
obtained a minimal separating set for \(2\times2\) quaternionic matrices.
Trace invariants for pairs of matrices in \(\SL(4,\C)\) and \(\SU(3,1)\)
were studied in \cite{gl-lawton}; in particular, a maximal algebraically
independent collection of fifteen trace functions was obtained for
\(\SL(4,\C)\). Since \(\SL(2,\H)\cong\SU^*(4)\) is a real form of
\(\SL(4,\C)\), it is natural to consider the corresponding trace problem
for \(\SL(2,\H)\). In our setting, we consider fifteen specific words in \(A\) and \(B\)
and the associated real parts of their quaternionic traces. We show that their
differentials are generically independent, so these fifteen trace functions
form local coordinates on a dense open subset of the moduli space of regular
loxodromic pairs. Our approach is local and geometric. 
We have not used algebraic invariant theory or seek a global set of trace invariants.

Finally, we turn to surface-group representations. We decompose the surface into pairs of pants and show that, on the generic locus, the map assigning to each pair of pants, the conjugacy classes of its boundary holonomies is a submersion. Consequently, after fixing the boundary conjugacy classes, the relative deformation space is locally six- dimensional. Moreover, this fiber admits six real-analytic local coordinates, which may be expressed locally as real-analytic functions of the fifteen trace coordinates above. Together with the boundary conjugacy data, these internal pants parameters and the three dimensional centralizer gluing freedom yield local Fenchel--Nielsen coordinates for the deformation space.

Kim \cite{ki1,ki2} studied deformations of parabolic
thrice-punctured sphere groups acting on real hyperbolic \(4\)-space.
Although the target groups and the types of peripheral holonomy are
different, these results provide a useful comparison with the
six-dimensional relative deformation space obtained here for a pair
of pants with fixed regular loxodromic boundary holonomies in
\(\SL(2,\H)\).

\subsection{Main Results}

Throughout, \(\Phi\colon\M(2,\H)\to\M(4,\C)\) denotes the standard complex embedding and
\[\tr(A):=\Re(a+d) \qquad \hbox{ for } \qquad A=\left(\begin{smallmatrix}a&b\\c&d\end{smallmatrix}\right).\]
Note that (see Lemma~\ref{lem:trace-is-complex-trace}) 
\[
\tr(A)=\tfrac12\,\Tr_\C\Phi(A),
\]
is a real-valued class function satisfying
$\tr(AB)=\tr(BA)$.  In particular, every quantity
written in terms of $\Tr_\C\Phi$ of words in $A$ is a conjugacy invariant.

Every loxodromic $A\in\SL(2,\H)$ is conjugate to $\mathrm{diag}(re^{I\theta},
r^{-1}e^{J\varphi})$ with $r>1$, $\theta,\varphi\in[0,\pi]$ and $I,J$ unit imaginary
quaternions.  We call $A$ \emph{regular} if $\theta,\varphi\in(0,\pi)$, i.e.\ if neither
eigenvalue class is real.  Write
\[
\nu(A):=r^{2}+r^{-2}.
\]

\begin{theorem}[Theorem~\ref{thm:lox-conj-class}]\label{thmA}
Let $\mathcal L$ be the set of conjugacy classes of loxodromic elements of $\SL(2,\H)$.
Then
\[
\Theta\colon\mathcal L\longrightarrow\R^3,\qquad
[A]\longmapsto\bigl(\tr(A),\tr(A^{-1}),\nu(A)\bigr)
\]
is injective, as is $[A]\mapsto\bigl(\tr(A),\tr(A^{-1}),\tr(A^2)\bigr)$; the two triples
determine one another.  Restricted to the regular loxodromic classes, $\Theta$ is a
real-analytic diffeomorphism onto the open set
\[
\mathcal D=\Bigl\{(x,y,z)\in\R^3:\ z>2,\ \
\Bigl|\frac{x+y}{\sqrt{z+2}}\pm\frac{x-y}{\sqrt{z-2}}\Bigr|<2\Bigr\}.
\]
\end{theorem}

The description of the image is the quaternionic analogue of the statement that a real
trace $|{\tr}|>2$ realises a hyperbolic element of $\SL(2,\R)$; it is what makes the
invariants usable as coordinates rather than merely as separating functions.

For pairs we prove a normal form theorem.  Call $(A,B)$ a \emph{loxodromic pair} if both
are loxodromic and their fixed point sets are disjoint.

\begin{theorem}[Theorem~\ref{thm:pair-normal-form}]\label{thmB}
Every regular loxodromic pair is conjugate in $\SL(2,\H)$ to a pair
\[
A_0=\begin{pmatrix} re^{I\theta}&0\\0&r^{-1}e^{J\varphi}\end{pmatrix},
\qquad
B_0=P\begin{pmatrix} se^{K\theta'}&0\\0&s^{-1}e^{L\varphi'}\end{pmatrix}P^{-1},
\qquad
P=\frac{1}{\sqrt{|1-w|}}\begin{pmatrix}1&w\\1&1\end{pmatrix},
\]
with $r,s>1$, $\theta,\varphi,\theta',\varphi'\in(0,\pi)$, $I,J,K,L$ unit imaginary
quaternions, and $w\in(\R+\R i)\setminus\{0,1\}$ with $\Imm(w)\ge0$.  The data
$(r,\theta,\varphi,s,\theta',\varphi',w)$ are uniquely determined, and $(I,J,K,L)$ is
determined up to simultaneous conjugation by $\Cent_{\H^{*}}(w)\cap\mathbb S^{3}$, a
circle when $w\notin\R$.  In this normal form the cross ratio of the four fixed points
equals $w$.
\end{theorem}
We denote the simultaneous conjugacy class of $(A,B)$ by $[(A,B)]$.
We write $\Mreg$ for the set of simultaneous conjugacy classes of regular
loxodromic pairs, equivalently, for the orbit space of the diagonal
conjugation action of $\SL(2,\H)$.

Consequently, $\Mreg$ contains an open dense generic locus $\Mreg_{*}$ which is a
$15$-dimensional real-analytic manifold, with
\[
15=3+3+2+7,
\]
corresponding respectively to the spectral data of $A$ and $B$, the
quaternionic cross-ratio, and the axial data
(Corollary~\ref{cor:dim15}). The next result shows that this local
$15$-dimensional structure can also be described entirely by trace
functions.

\begin{theorem}[Theorem~\ref{thm:word-traces}]\label{thmC}
Let
\[
\begin{aligned}
\mathcal W=\bigl(&
A,\ A^{-1},\ A^2,\ B,\ B^{-1},\ B^2,\ AB,\ AB^{-1},\
A^{-1}B,\ A^{-1}B^{-1},\\
& A^2B,\ AB^2,\ A^2B^2,\ ABAB^{-1},\
ABA^{-1}B^{-1}
\bigr).
\end{aligned}
\]
For each $W\in\mathcal W$, consider the real-valued trace function
\[
[(A, B)]\longmapsto \tr\bigl(W(A,B)\bigr).
\]
Then, on a dense open subset of $\Mreg_{*}$, these fifteen trace functions
have linearly independent differentials. Consequently,
\[
\mathcal T:\Mreg_{*}\longrightarrow\R^{15},
\qquad
[(A, B)]\longmapsto
\bigl(\tr(W(A,B))\bigr)_{W\in\mathcal W},
\]
is a local real-analytic diffeomorphism.
\end{theorem}
If $A$ is regular loxodromic with eigenvalues
$re^{I\theta},r^{-1}e^{J\varphi}$ then
\[
\Cent_{\SL(2,\H)}(A)\ \cong\ \R_{>0}\times\Uni(1)\times\Uni(1),
\]
a three-dimensional group consisting of one translation along the axis of $A$ and two
independent rotations about it (Proposition~\ref{prop:centralizer}).  The twist
parameter attached to a gluing curve is therefore three-dimensional.  Combining the
three ingredients:  lengths, twists, and the internal data  of a pair of pants, 
gives the dimension count underlying a quaternionic Fenchel--Nielsen theory.
\begin{definition}\label{def:generic-representation}
Let $\rho:\pi_1(S)\longrightarrow\SL(2,\mathbb H)$ be a representation such
that $\rho(\gamma_m)$ is regular loxodromic for every curve $\gamma_m$ in a
fixed pants decomposition
$\mathcal C=\{\gamma_1,\ldots,\gamma_{3g-3}\}$.

For each pair of pants $Y_j$, choose oriented boundary loops
$\alpha_j,\beta_j,\delta_j$, based in the usual way, so that
$\alpha_j\beta_j\delta_j=1$ in $\pi_1(Y_j)$.  Thus $\alpha_j,\beta_j$ form a free
basis of $\pi_1(Y_j)$.  Write
$A_j=\rho(\alpha_j)$,
$B_j=\rho(\beta_j)$, and
$C_j=\rho(\delta_j)=(A_jB_j)^{-1}$.
We call $\rho$ \emph{generic with respect to $\mathcal C$} if, for every pair of pants
$Y_j$,
\begin{equation}\label{eq:generic}
\mathfrak z(A_j)\cap\mathfrak z(B_j)=0
\qquad\text{and}\qquad
\Fix(A_j)\cap\Fix(B_j)=\emptyset ,
\end{equation}
where $\mathfrak z(g)=\ker(1-\operatorname{Ad}_g)\subset\mathfrak{sl}(2,\mathbb H)$ and
$\Fix(g)\subset\hH$ is the fixed point set of $g$.
\end{definition}

The first condition in \eqref{eq:generic} says that the simultaneous
infinitesimal centralizer of \(A_j\) and \(B_j\) is zero; it is what
makes the boundary constraint transverse. In particular, it excludes
pants representations contained in a common complex slice. Indeed, if
\(A_j\) and \(B_j\) have entries in
\(\C_I=\R+\R I\), then
\(Y_I=diag(I,I)\in\mathfrak{sl}(2,\H)\) commutes with both, and hence
\[
\R Y_I\subset \mathfrak z(A_j)\cap\mathfrak z(B_j)\neq0.
\]
Thus even a complex-slice representation with regular loxodromic
boundary holonomies is not generic in the sense of
Definition~\ref{def:generic-representation}. The second condition says
that \((A_j,B_j)\) is a loxodromic pair in the sense of
Theorem~\ref{thmB}; it is what makes the simultaneous conjugation
action proper with finite stabilisers. Both are open conditions, and
both are needed: neither follows from the other. Fuchsian
representations are excluded already by regularity, since hyperbolic
elements of \(\SL(2,\R)\subset\SL(2,\H)\) have real eigenvalue classes
and hence are not regular loxodromic in the sense of
Definition~\ref{def:lox}. Thus the local theory below concerns a
genuinely quaternionic generic locus.
\begin{theorem}\label{thm:FN-count}
Let $S$ be a closed surface of genus $g\geq 2$, and let
$\mathcal C=\{\gamma_1,\ldots,\gamma_{3g-3}\}$ be a pants decomposition
of $S$, with pairs of pants $Y_1,\ldots,Y_{2g-2}$.  Let
$\rho:\pi_1(S)\longrightarrow\SL(2,\mathbb H)$ be a representation such that the
restriction of $\rho$ to every curve $\gamma_m\in\mathcal C$ is regular loxodromic, and
assume that $\rho$ is generic with respect to $\mathcal C$.  Then a neighbourhood of
$[\rho]$ in the local deformation space modulo conjugation is a real-analytic orbifold
of real dimension $30g-30$.  Along the pants decomposition the local parameter count
splits as
\[
30g-30
=\underbrace{(3g-3)\cdot3}_{\text{lengths}}
+\underbrace{(3g-3)\cdot3}_{\text{twists}}
+\underbrace{(2g-2)\cdot6}_{\text{internal pants moduli}}.
\]
More precisely:

\begin{enumerate}
\item[(1)] \emph{(Lengths.)} The conjugacy class of $\rho(\gamma_m)$ is exactly the point
$\Theta\bigl([\rho(\gamma_m)]\bigr)\in\mathcal D$ of Theorem~\ref{thm:lox-conj-class},
and every point of $\mathcal D$ occurs.  Hence the common boundary conjugacy classes
along $\mathcal C$ contribute $3(3g-3)$ parameters, three for each curve.

\item[(2)] \emph{(Twists.)} For each curve $\gamma_m$, the centralizer
\(
\Cent_{\PSL(2,\mathbb H)}\bigl([\rho(\gamma_m)]\bigr)
\)
is three-dimensional. It parametrizes locally the possible re-gluings after
the two boundary holonomies have been identified. Hence the gluings contribute
$3(3g-3)$ parameters.

\item[(3)] \emph{(Internal pants data.)} For each pair of pants \(Y_n\), fix the conjugacy classes \([A],[B],[C]\) of its boundary holonomies, with \(C=(AB)^{-1}\). Let 
\[
\mathcal M_{Y_n}\bigl([A],[B],[C]\bigr)=
\left\{
(A',B')\in\SL(2,\mathbb H)^2:
A'\in[A],\;
B'\in[B],\;
(A'B')^{-1}\in[C]
\right\}\big/\SL(2,\mathbb H), 
\]
where the quotient is taken with respect to simultaneous conjugation. Near the point determined by \(\rho\), this is a six- dimensional (12+12-3-15=6) real-analytic orbifold. 

Thus, the \(2g-2\) pairs of pants contribute to  \(6(2g-2)\) internal parameters which may locally be chosen as real-analytic functions of the trace coordinates in Theorem~\ref{thm:word-traces}.
\end{enumerate}
\end{theorem} 

The six local degrees of freedom in
\(\mathcal M_{Y_n}\bigl([A],[B],[C]\bigr)\) will be called the
\emph{internal pants parameters}. Thus, after fixing the three boundary
conjugacy classes, the remaining local deformations of the pair of pants
are described by six internal parameters.

\begin{remark}
The dimension count in item~(3) can be explained as follows. A regular
loxodromic element has a \(3\)-dimensional centralizer in
\(\SL(2,\H)\), so each of the conjugacy classes \([A]\) and \([B]\) has
dimension \(15-3=12\). The condition
\((A'B')^{-1}\in[C]\) has codimension \(3\), and the required
transversality is established in Lemma~\ref{lem:pants-transversality}
below. Finally, quotienting by simultaneous conjugation removes
\(15\) dimensions. Hence the relative deformation space has dimension
\(
12+12-3-15=6.
\)
\end{remark}

\begin{remark}\label{rem:count-not-coordinates}
Theorem~\ref{thm:FN-count} gives a local dimension decomposition rather
than a preferred global coordinate system. The proof in
Section~\ref{sec:FN-proof} identifies the three types of local
parameters: boundary conjugacy classes, internal pants parameters, and
centralizer gluing parameters. The construction of global
Fenchel--Nielsen coordinates requires additional work, as discussed in
Section~\ref{sec:open}.
\end{remark}

\begin{remark}\label{rem:goldman}
The total dimension \(30g-30\) agrees with Goldman's general
deformation-theoretic formula for surface-group representations with
finite infinitesimal centralizer; see~\cite{goldman}. This provides an
independent consistency check. The proof of
Theorem~\ref{thm:FN-count}, however, is entirely based on the
pants-and-gluing construction and does not use cohomological methods.
\end{remark}

\begin{remark}\label{rem:pants-not-rigid}
The six-dimensional internal deformation space in item~(3) is a
distinctive feature of the quaternionic setting. For
\(G=\SL(2,\R)\), the corresponding dimension count is
\((2+2-1-3)=0,
\)
while for \(G=\SL(2,\C)\) it is
\(
(4+4-2-6)=0.
\)
Thus, in these classical cases, a pair of pants with prescribed boundary
conjugacy classes is locally rigid, and the Fenchel--Nielsen parameters
come from boundary data and twists.

For \(G=\SU(2,1)\), the analogous dimension count gives
\(
(6+6-2-8)=2.
\)
Thus, after fixing the three boundary conjugacy classes, two internal
degrees of freedom remain. Additional parameters of this type occur in
the complex hyperbolic Fenchel--Nielsen theory of
Parker--Platis~\cite{parkplat}. Analogous additional data occur in the
quaternionic hyperbolic setting of~\cite{kgsk}. For
\(G=\SL(2,\H)\), the relative deformation space of a pair of pants with
prescribed regular loxodromic boundary conjugacy classes is
six-dimensional. Thus any Fenchel--Nielsen-type parametrisation must
account for six internal parameters for each pair of pants.
\end{remark}

\subsection{Organisation}
Section~\ref{sec:prelim} recalls the basic facts about \(\mathbb H\),
\(\SL(2,\mathbb H)\) and the complex embedding, and develops the real trace and the
characteristic polynomial.  Section~\ref{sec:lox} classifies loxodromic conjugacy
classes, while Section~\ref{sec:centralizer} computes their centralizers.
Section~\ref{sec:crossratio} studies the quaternionic cross ratio, and
Section~\ref{sec:pairs} gives the normal form and local coordinates for loxodromic
pairs.  Section~\ref{sec:frames} introduces the framed picture, which isolates the axis
data as the genuinely quaternionic part of the classification.  Section~\ref{sec:FN}
proves Theorem~\ref{thm:FN-count}.  In Section~\ref{sec:open} we discuss what
remains to be done before these parameters become a complete Fenchel--Nielsen
coordinate system. 
\section{Preliminaries}\label{sec:prelim}
We collect in this section the basic definitions and preliminary facts that
will be used throughout the paper. Standard references for the quaternionic
M\"obius group and its loxodromic elements include
\cite{cg,foreman, kg, ps,wi}.

\subsection{Quaternions}

Let $\H=\R\oplus\R i\oplus\R j\oplus\R k$ be the division algebra of Hamilton
quaternions.  For $q=q_0+q_1i+q_2j+q_3k$ write $\Re(q)=q_0$, $\Imm(q)=q-\Re(q)$,
$\bar q=\Re(q)-\Imm(q)$ and $|q|^2=q\bar q$.  Let
\[
\S2=\{I\in\H:\ \Re(I)=0,\ |I|=1\}=\{I\in\H:\ I^2=-1\}
\]
be the two-sphere of unit imaginary quaternions.  For $I\in\S2$ put
$\C_I=\R+\R I$; this is a subfield of $\H$ isomorphic to $\C$, and
$\H=\bigcup_{I\in\S2}\C_I$ with $\C_I\cap\C_{I'}=\R$ for $I'\ne\pm I$.

Two quaternions $q,q'$ are \emph{similar}, written $q\sim q'$, if $q'=pqp^{-1}$ for some
$p\in\H^{*}$.  We record the standard fact that we shall use constantly.

\begin{lemma}\label{lem:similarity}
$q\sim q'$ if and only if $\Re(q)=\Re(q')$ and $|q|=|q'|$.  Every $q\in\H\setminus\R$
may be written uniquely as
\[
q=re^{I\theta}=r(\cos\theta+I\sin\theta),\qquad
r=|q|>0,\quad \theta\in(0,\pi),\quad I=\frac{\Imm(q)}{|\Imm(q)|}\in\S2 ,
\]
and the similarity class of $q$ is the two-sphere $\{re^{I'\theta}:I'\in\S2\}$.  The
centralizer of $q\notin\R$ in $\H^{*}$ is $\C_I^{*}$, a real two-dimensional group;
the centralizer of a real quaternion is all of $\H^{*}$.
\end{lemma}

We call $I$ in Lemma~\ref{lem:similarity} the \emph{axis} of $q$ and $\theta$ its
\emph{argument}.  The similarity class of $q$ remembers $r$ and $\theta$ but forgets the
axis; conversely, choosing a representative of a similarity class is the same as
choosing a point of $\S2$.  This elementary remark is the source of all the extra
parameters that distinguish the quaternionic theory from the complex one.

\subsection{The complex embedding}
Writing $q=z+wj$ with $z,w\in\C=\C_i$, the map
\begin{equation}\label{eq:Phi-quat}
\Phi\colon\H\longrightarrow\M(2,\C),\qquad
\Phi(z+wj)=\begin{pmatrix} z& w\\ -\bar w& \bar z\end{pmatrix}
\end{equation}
is an injective homomorphism of real algebras with $\Phi(\bar q)=\Phi(q)^{*}$,
$\det\Phi(q)=|q|^{2}$ and $\Tr\Phi(q)=2\Re(q)$.  Applying \eqref{eq:Phi-quat} entrywise
gives an injective homomorphism of real algebras
\[
\Phi\colon\M(2,\H)\longrightarrow\M(4,\C),\qquad \Phi(A^{*})=\Phi(A)^{*},
\]
where $A^{*}=\bar A^{\,t}$ is the quaternionic conjugate transpose.

\subsection{The group \(\SL(2,\H)\)}
Let \(\ddet_{\mathrm D}\) denote the Dieudonn\'e determinant on $\GL(2,\H)$, the
surjective homomorphism onto $\R_{>0}$ characterized by
\[
\ddet_{\mathrm D}(A)=\bigl(\det\nolimits_\C\Phi(A)\bigr)^{1/2},
\]
and given explicitly by
$\ddet_{\mathrm D}\left(\begin{smallmatrix}a&b\\c&d\end{smallmatrix}\right)=|ad-aca^{-1}b|$
when $a\ne0$, and by $|cb|$ when $a=0$. When there is no ambiguity about the quaternionic matrix under consideration, we shall simply omit the subscript \(D\) and write \(\det A\) for \(\det_D A\).

In particular
$\ddet_{\mathrm D}\,\mathrm{diag}(\alpha,\delta)=|\alpha||\delta|$ and
$\ddet_{\mathrm D}\left(\begin{smallmatrix}1&w\\1&1\end{smallmatrix}\right)=|1-w|$.  

The group we consider in this paper is: 
\[
\SL(2,\H)=\{A\in\GL(2,\H):\ \ddet_{\mathrm D}(A)=1\}.
\]
This is a connected real Lie group of dimension $15$, isomorphic to
$\mathrm{Spin}(5,1)$ and one has $\det_\C\Phi(A)=\ddet_{\mathrm D}(A)^{2}$, so
$\Phi\bigl(\SL(2,\H)\bigr)\subset\SL(4,\C)$.  See \cite{cpw,ps,kg, pa2} for details.

$\SL(2,\H)$ acts on $\hH=\H\cup\{\infty\}$ by quaternionic M\"obius transformations
\[
A\cdot z=(az+b)(cz+d)^{-1},\qquad
A=\begin{pmatrix}a&b\\c&d\end{pmatrix},
\]
the kernel of the action being the centre $\{\pm I\}$.  The quotient
$\PSL(2,\H)=\SL(2,\H)/\{\pm I\}$ is isomorphic to $\SO^{+}(5,1)\cong\Isom^{+}(\mathbf
H^{5}_{\R})$, and $\hH\cong\mathbb S^{4}$ is the ideal boundary of $\mathbf H^{5}_\R$.
The action of $\PSL(2,\H)$ on ordered triples of distinct points of $\hH$ is transitive. The stabiliser of
$(\infty,0,1)$ is
\begin{equation}\label{eq:stab-triple}
\mathrm{Stab}(\infty,0,1)=\{z\mapsto qzq^{-1}:q\in\H^{*}\}\cong\SO(3).
\end{equation}

\subsection{Right eigenvalues, eigenlines and loxodromic elements}
We regard $\H^{2}$ as a right $\H$-module of column vectors, so that $\M(2,\H)$ acts on
the left.  A quaternion $\lambda$ is a \emph{right eigenvalue} of $A$ if $Av=v\lambda$
for some $v\ne0$.  Since $A(vq)=(vq)(q^{-1}\lambda q)$, right eigenvalues occur in
similarity classes.  If $\lambda$ is a right eigenvalue of $A$ then the two complex
numbers $re^{\pm i\theta}$ (where $\lambda\sim re^{i\theta}$) are eigenvalues of
$\Phi(A)$, and every eigenvalue of $\Phi(A)$ arises this way.

\begin{definition}\label{def:lox}
$A\in\SL(2,\H)$ is \emph{loxodromic} if the induced M\"obius transformation of $\hH$ has
exactly two fixed points, one attracting and one repelling.  Equivalently
\cite{cpw,ps,kg}, $A$ is conjugate in $\SL(2,\H)$ to
\begin{equation}\label{eq:lox-normal}
D=\begin{pmatrix} re^{I\theta}&0\\ 0& r^{-1}e^{J\varphi}\end{pmatrix},
\qquad r>1,\quad \theta,\varphi\in[0,\pi],\quad I,J\in\S2 .
\end{equation}
We call $A$ \emph{regular loxodromic} if moreover $\theta,\varphi\in(0,\pi)$, i.e.\ if
neither right eigenvalue class is real.
\end{definition}

Note that the two eigenvalue similarity classes of a loxodromic element are
automatically distinct, indeed have different moduli $r\ne r^{-1}$; hence a loxodromic
element is always diagonalisable with non-similar eigenvalues.  Regularity in the sense
of Definition~\ref{def:lox} is the condition that the centralizer of $A$ has minimal
dimension; see Proposition~\ref{prop:centralizer}.

\begin{lemma}\label{lem:diag-conj}
Let $\lambda,\mu,\lambda',\mu'\in\H^{*}$ with $|\lambda\mu|=|\lambda'\mu'|=1$.  Then
$\mathrm{diag}(\lambda,\mu)$ and $\mathrm{diag}(\lambda',\mu')$ are conjugate in
$\SL(2,\H)$ if and only if either ($\lambda\sim\lambda'$ and $\mu\sim\mu'$) or
($\lambda\sim\mu'$ and $\mu\sim\lambda'$).
\end{lemma}

\begin{proof}
If $\lambda'=q\lambda q^{-1}$ and $\mu'=p\mu p^{-1}$, replace $q,p$ by $q/|q|,p/|p|$,
which does not change the conjugations; then $g=\mathrm{diag}(q,p)$ lies in $\SL(2,\H)$
and $gDg^{-1}=\mathrm{diag}(\lambda',\mu')$.  The permutation matrix
$\left(\begin{smallmatrix}0&1\\1&0\end{smallmatrix}\right)$ lies in $\SL(2,\H)$ and
exchanges the two diagonal entries.  Conversely, conjugate matrices have the same
right eigenvalue similarity classes with multiplicity.
\end{proof}

\subsection{The real trace and the characteristic polynomial}\label{sec:trace}

There is no quaternion-valued trace on $\M(2,\H)$ that is similarity invariant. However, one can define the following that is similarity invariant. 

\begin{definition}
For $A=\left(\begin{smallmatrix}a&b\\c&d\end{smallmatrix}\right)\in\M(2,\H)$ set
\[
\tr(A):=\Re(a+d)=\Re(a)+\Re(d)\in\R .
\]
\end{definition}

\begin{lemma}\label{lem:trace-is-complex-trace}
For every $A\in\M(2,\H)$,
\[
\tr(A)=\tfrac12\,\Tr_\C\Phi(A).
\]
Consequently $\tr$ is $\R$-linear, $\tr(AB)=\tr(BA)$, $\tr(gAg^{-1})=\tr(A)$ for all
$g\in\GL(2,\H)$, and $\tr(A^{*})=\tr(A)$.
\end{lemma}

\begin{proof}
$\Phi(A)$ is the $4\times4$ matrix with diagonal blocks $\Phi(a)$ and $\Phi(d)$, so
$\Tr_\C\Phi(A)=\Tr\Phi(a)+\Tr\Phi(d)=2\Re(a)+2\Re(d)$.  The remaining assertions follow
from the corresponding properties of $\Tr_\C$ together with the facts that $\Phi$ is a
homomorphism of real algebras and $\Phi(A^{*})=\Phi(A)^{*}$.
\end{proof}

\begin{definition}
The \emph{characteristic polynomial} of $A\in\SL(2,\H)$ is
\[
\chi_A(x):=\det\nolimits_\C\bigl(xI_4-\Phi(A)\bigr)=x^{4}-c_3x^{3}+c_2x^{2}-c_1x+1 .
\]
\end{definition}

Note that $\chi_A$ has
real coefficients.  The spectrum of $\Phi(A)$ is invariant under complex
conjugation.

\begin{lemma}\label{lem:char-poly}
Let $A\in\SL(2,\H)$ be loxodromic, conjugate to $D$ as in \eqref{eq:lox-normal}.  Then
\[
\chi_A(x)=\bigl(x^{2}-2r\cos\theta\,x+r^{2}\bigr)
          \bigl(x^{2}-2r^{-1}\cos\varphi\,x+r^{-2}\bigr),
\]
so that
\begin{equation}\label{eq:c-in-rthetaphi}
c_3=2\bigl(r\cos\theta+r^{-1}\cos\varphi\bigr),\quad
c_2=r^{2}+r^{-2}+4\cos\theta\cos\varphi,\quad
c_1=2\bigl(r^{-1}\cos\theta+r\cos\varphi\bigr).
\end{equation}
Furthermore
\begin{equation}\label{eq:c-in-traces}
c_3=2\tr(A),\qquad c_1=2\tr(A^{-1}),\qquad c_2=2\tr(A)^{2}-\tr(A^{2}).
\end{equation}
\end{lemma}

\begin{proof}
Since $\Phi$ is a homomorphism, $\Phi(D)$ is block diagonal with blocks
$\Phi(re^{I\theta})$ and $\Phi(r^{-1}e^{J\varphi})$.  By \eqref{eq:Phi-quat} the first
block has trace $2\Re(re^{I\theta})=2r\cos\theta$ and determinant $|re^{I\theta}|^{2}=r^{2}$,
so its characteristic polynomial is $x^{2}-2r\cos\theta\,x+r^{2}$; likewise the second
block has characteristic polynomial $x^{2}-2r^{-1}\cos\varphi\,x+r^{-2}$.  Expanding the
product gives \eqref{eq:c-in-rthetaphi}.

For \eqref{eq:c-in-traces}, write $e_1,\dots,e_4$ for the elementary symmetric functions
of the eigenvalues of $\Phi(A)$, so $c_3=e_1$, $c_2=e_2$, $c_1=e_3$ and $e_4=1$.  By
Lemma~\ref{lem:trace-is-complex-trace}, $e_1=\Tr_\C\Phi(A)=2\tr(A)$, and
$e_3=e_4\cdot\Tr_\C\Phi(A)^{-1}=2\tr(A^{-1})$.  Newton's identity
$p_2=e_1^{2}-2e_2$ with $p_2=\Tr_\C\Phi(A^{2})=2\tr(A^{2})$ gives
$c_2=e_2=\tfrac12(4\tr(A)^{2}-2\tr(A^{2}))=2\tr(A)^{2}-\tr(A^{2})$.
\end{proof}

Note that \eqref{eq:c-in-traces} holds for all $A\in\SL(2,\H)$, not only loxodromic ones.

\subsection{The quaternionic cross ratio}\label{sec:crossratio}

We recall  information on the  quaternionic cross ratio in this section. We follow the construction of  Gwynne and Libine \cite{gl}. 
\begin{definition}\label{def:crossratio}
For pairwise distinct $z_1,z_2,z_3,z_4\in\H$ set
\begin{equation}\label{eq:crossratio}
\X(z_1,z_2,z_3,z_4):=(z_1-z_3)(z_2-z_3)^{-1}(z_2-z_4)(z_1-z_4)^{-1}\in\H^{*},
\end{equation}
extended to $\hH$ by the usual limiting conventions when one $z_i=\infty$.
\end{definition}

The order of the four factors in \eqref{eq:crossratio} matters, and the apparently
different expression $(z_3-z_1)(z_3-z_2)^{-1}(z_4-z_2)(z_4-z_1)^{-1}$ used elsewhere in
the literature is equal to \eqref{eq:crossratio}, since the four sign changes cancel.

\begin{proposition}\label{prop:crossratio-similarity} \cite{gl}
For every $g\in\PSL(2,\H)$ there is $q\in\H^{*}$, depending on $g$ and on the quadruple,
with
\[
\X(gz_1,gz_2,gz_3,gz_4)=q\,\X(z_1,z_2,z_3,z_4)\,q^{-1}.
\]
Consequently $|\X|$ and $\Re\X$ are $\PSL(2,\H)$-invariants of ordered quadruples of
distinct points of $\hH$.
\end{proposition}

\begin{proposition}\label{prop:crossratio-complete} \cite{gl} 
Let $(z_1,z_2,z_3,z_4)$ be an ordered quadruple of distinct points of $\hH$.  There is a
unique $w\in\H\setminus\{0,1\}$ with $\Imm(w)\in\R_{\ge0}\,i$ such that
$(z_1,z_2,z_3,z_4)$ is $\PSL(2,\H)$-equivalent to $(\infty,0,1,w)$, and then
\[
\X(\infty,0,1,w)=w .
\]
Two ordered quadruples of distinct points of $\hH$ are $\PSL(2,\H)$-equivalent if and
only if their cross ratios have the same modulus and the same real part.  Finally, the
four points lie on a common circle or line in $\hH$ if and only if $\X\in\R$.
\end{proposition}
\begin{definition}\label{def:pair-crossratio}
Let $(A,B)$ be a pair of loxodromic elements of $\SL(2,\H)$ with disjoint fixed point
sets.  Let $a_A,r_A$ (respectively $a_B,r_B$) be the attracting and repelling fixed
points of $A$ (respectively $B$).  The \emph{cross ratio of the pair} is
\[
\X(A,B):=\X(a_A,r_A,a_B,r_B),
\]
well defined up to similarity; $|\X(A,B)|$ and $\Re\X(A,B)$ are invariants of the
simultaneous conjugacy class.  The pair is called \emph{degenerate} if its four fixed
points lie on a common circle or line, equivalently (by
Proposition~\ref{prop:crossratio-complete}) if $\X(A,B)\in\R$.
\end{definition}

The non-degenerate locus is open and dense.  One must, however, be careful with its
codimension in the moduli space of pairs.  In the normalized quadruple space the
condition $w\in\R$ is a boundary condition on the two-dimensional cross-ratio parameter,
but at the same time the stabiliser of $(\infty,0,1,w)$ jumps from a circle to
$\SO(3)$.  Consequently, on the open part of the degenerate locus on which this
$\SO(3)$-action on the four axes has finite stabiliser, the normal form has dimension
\[
(2+4+8+1)-3=12.
\]
Thus the generic degenerate locus has codimension $3$ inside the $15$-dimensional generic
pair moduli.  This does not assert that the ambient moduli space is singular at every
degenerate pair; when the simultaneous centralizer of $(A,B)$ is finite,
Lemma~\ref{lem:proper-action} still gives a $15$-dimensional local orbifold.

\section{Conjugacy classes of loxodromic elements}\label{sec:lox}
The characteristic polynomial of a loxodromic element determines the conjugacy class. The following proposition is well-known, cf. \cite{cg}, \cite{kg}, \cite{gkl},  \cite{ps}. 

\begin{proposition}\label{prop:charpoly-determines}
Two loxodromic elements of $\SL(2,\H)$ are conjugate if and only if they have the same
characteristic polynomial.  More precisely, $\chi_A$ determines the ordered pair of
right eigenvalue similarity classes of a loxodromic $A$.
\end{proposition}

\begin{remark}\label{rem:charpoly-scope}
Proposition~\ref{prop:charpoly-determines} is what makes the whole scheme work, and it is
worth being precise about why, and about how far it extends.

Each right eigenvalue similarity class $\lambda\sim re^{i\theta}$ of $A$ contributes the
conjugate pair $re^{\pm i\theta}$ to the spectrum of $\Phi(A)$, so any grouping of the
four roots of $\chi_A$ into two admissible quadratic factors must consist of two
conjugation-closed pairs, and each such factor $x^{2}-2r\cos\theta\,x+r^{2}$ has positive
constant term $r^{2}$.  For a loxodromic element the two moduli $r$ and $r^{-1}$ are
distinct, and this pins down not only the grouping but also which factor belongs to the
attracting class; this is what lets $\chi_A$ recover the \emph{ordered} pair of
eigenvalue classes.  For an elliptic element the same constraints still force a unique
grouping into conjugation-closed pairs with positive constant term, so $\chi_A$ again
determines the two eigenvalue classes, but now only as an unordered pair which by
Lemma~\ref{lem:diag-conj} is exactly the conjugacy class.

Where $\chi_A$ genuinely fails is for non-semisimple elements.  The identity and
$\left(\begin{smallmatrix}1&1\\0&1\end{smallmatrix}\right)$ both have characteristic
polynomial $(x-1)^{4}$ and are not conjugate, so $\chi_A$ does not separate parabolic and
screw-parabolic classes from elliptic ones.  Everything below is therefore stated for
loxodromic elements, where diagonalisability is automatic
(Definition~\ref{def:lox}).
\end{remark}

\subsection{The three real invariants}

\begin{definition}\label{def:nu}
For $A\in\SL(2,\H)$ loxodromic with normal form \eqref{eq:lox-normal}, set
\[
\nu(A):=r^{2}+r^{-2},
\]
where $r>1$ is the modulus of the attracting right eigenvalue class.  By
Lemma~\ref{lem:diag-conj}, $\nu$ is a conjugacy invariant.
\end{definition}

\begin{theorem}\label{thm:lox-conj-class}
Let $\mathcal L$ denote the set of conjugacy classes of loxodromic elements of
$\SL(2,\H)$, and let $\mathcal L^{\mathrm{reg}}\subset\mathcal L$ be the subset of
regular classes.  Write
\[
x=\tr(A),\qquad y=\tr(A^{-1}),\qquad z=\nu(A).
\]
\begin{enumerate}
\item[(1)] The map $\Theta\colon\mathcal L\to\R^{3}$, $[A]\mapsto(x,y,z)$, is injective.
\item[(2)] The coefficients of $\chi_A$ are recovered by $c_3=2x$, $c_1=2y$ and
\begin{equation}\label{eq:c2-formula}
c_2=z+\frac{4\bigl(xyz-x^{2}-y^{2}\bigr)}{z^{2}-4}
   =z+\frac{4}{z-2}\Bigl(xy-\frac{(x+y)^{2}}{z+2}\Bigr).
\end{equation}
Equivalently $\tr(A^{2})=2x^{2}-c_2$, so $[A]\mapsto(\tr A,\tr A^{-1},\tr A^{2})$ is
also injective on $\mathcal L$, and the two triples determine one another.
\item[(3)] $\Theta$ maps $\mathcal L^{\mathrm{reg}}$ real-analytically and
bijectively onto the open set
\[
\mathcal D:=\Bigl\{(x,y,z)\in\R^{3}\ :\ z>2,\ \ |u|<1,\ |v|<1\Bigr\},
\]
where
\[
u=\frac12\Bigl(\frac{x+y}{\sqrt{z+2}}+\frac{x-y}{\sqrt{z-2}}\Bigr),\qquad
v=\frac12\Bigl(\frac{x+y}{\sqrt{z+2}}-\frac{x-y}{\sqrt{z-2}}\Bigr),
\]
with inverse determined by $r=\tfrac12\bigl(\sqrt{z+2}+\sqrt{z-2}\bigr)$,
$\cos\theta=u$, $\cos\varphi=v$.  In particular $\Theta|_{\mathcal L^{\mathrm{reg}}}$ is
a homeomorphism onto $\mathcal D$ and every point of $\mathcal D$ is realised.
\end{enumerate}
\end{theorem}

\begin{proof}
Put $u=\cos\theta$ and $v=\cos\varphi$ for the normal form \eqref{eq:lox-normal}.  By
Lemma~\ref{lem:char-poly} and Lemma~\ref{lem:trace-is-complex-trace},
\begin{equation}\label{eq:xyz}
x=ru+r^{-1}v,\qquad y=r^{-1}u+rv,\qquad z=r^{2}+r^{-2}.
\end{equation}
Since $r>1$ we have $z>2$ and $r$ is recovered from $z$ by
$r+r^{-1}=\sqrt{z+2}$, $r-r^{-1}=\sqrt{z-2}$, whence
$r=\tfrac12(\sqrt{z+2}+\sqrt{z-2})$.  Adding and subtracting the first two equations of
\eqref{eq:xyz},
\[
x+y=(r+r^{-1})(u+v)=\sqrt{z+2}\,(u+v),\qquad
x-y=(r-r^{-1})(u-v)=\sqrt{z-2}\,(u-v),
\]
so that $u$ and $v$ are given by the displayed formulas in (3).  Thus $(x,y,z)$
determines $(r,u,v)$, hence by Lemma~\ref{lem:similarity} the two eigenvalue similarity
classes, hence by Lemma~\ref{lem:diag-conj} the conjugacy class.  This proves (1).

For (2), from the two displayed identities,
\[
(u+v)^{2}=\frac{(x+y)^{2}}{z+2},\qquad (u-v)^{2}=\frac{(x-y)^{2}}{z-2},
\]
so
\[
4uv=(u+v)^{2}-(u-v)^{2}=\frac{(x+y)^{2}}{z+2}-\frac{(x-y)^{2}}{z-2}
=\frac{4\bigl(xyz-x^{2}-y^{2}\bigr)}{z^{2}-4},
\]
the last step by direct expansion.  Since $c_2=4uv+z$ by \eqref{eq:c-in-rthetaphi}, this
is \eqref{eq:c2-formula}; the second expression in \eqref{eq:c2-formula} is the same
rational function, as one checks by clearing denominators.  The identity
$\tr(A^{2})=2x^{2}-c_2$ is \eqref{eq:c-in-traces}.  Injectivity of
$[A]\mapsto(\tr A,\tr A^{-1},\tr A^{2})$ now follows from
Proposition~\ref{prop:charpoly-determines}.

For (3), the map
\[
F\colon(1,\infty)\times(-1,1)^{2}\longrightarrow\R^{3},\qquad
F(r,u,v)=\bigl(ru+r^{-1}v,\ r^{-1}u+rv,\ r^{2}+r^{-2}\bigr)
\]
is real-analytic, and by Lemma~\ref{lem:similarity} the assignment
$(r,u,v)\mapsto[\mathrm{diag}(re^{i\arccos u},r^{-1}e^{i\arccos v})]$ is a bijection onto
$\mathcal L^{\mathrm{reg}}$.  Its Jacobian determinant is
\[
\det DF=\pm\,2\bigl(r-r^{-3}\bigr)\det\begin{pmatrix} r& r^{-1}\\ r^{-1}& r\end{pmatrix}
=\pm\,2\bigl(r-r^{-3}\bigr)\bigl(r^{2}-r^{-2}\bigr)\ne0
\]
for $r>1$, because the last row of $DF$ is $(2(r-r^{-3}),0,0)$.  Hence $F$ is a local
diffeomorphism; being injective by the computation above, it is a diffeomorphism onto
its open image, which is exactly $\mathcal D$ by the displayed inversion formulas.
\end{proof}

\begin{remark}\label{rem:realisation}
Part (3) contains the `realisation' statement that is needed if the invariants are to be
used as coordinates: every triple $(x,y,z)\in \mathcal D$ occurs as
$(\tr A,\tr A^{-1},\nu(A))$ for a regular loxodromic $A$, namely for
$A=\mathrm{diag}(re^{i\arccos u},r^{-1}e^{i\arccos v})$.  Injectivity alone would not
suffice.
\end{remark}

\begin{remark}\label{rem:trsquare-nonopen}
The three quantities $(\tr A,\tr A^{-1},\tr A^{2})$ can also be used to distinguish the classes, as shown in (2). However, the region they describe is less convenient to work with. For our coordinates, $\nu$ is better because it only needs to satisfy the simple condition $\nu>2$.
On the other hand, $\tr(A^{2})$ has a useful advantage: it is manifestly the trace of a
word in $A$, hence manifestly a class function.  The two coordinate systems are related
by
\[
\tr(A^{2})=2x^{2}-z-\frac{4(xyz-x^{2}-y^{2})}{z^{2}-4},
\]
and this change of variables is not merely locally invertible but globally monotone in
the last variable:
\begin{equation}\label{eq:trA2-monotone}
\frac{\partial\,\tr(A^{2})}{\partial z}\bigg|_{x,y\ \mathrm{fixed}}
\ \le\ -\,\frac{z-2}{z+2}\ <\ 0
\qquad\text{on }\mathcal D .
\end{equation}
Indeed, writing $t=r$ and $u=\cos\theta$, $v=\cos\varphi$ as in the proof of
Theorem~\ref{thm:lox-conj-class}, one has $\tr(A^{2})=t^{2}(2u^{2}-1)+t^{-2}(2v^{2}-1)$,
and holding $x,y$ fixed while differentiating in $t$ gives
\[
\frac{\partial\,\tr(A^{2})}{\partial t}
=-\frac{2(t^{4}-1)}{t^{3}}
+\frac{8\bigl(uv(t^{4}+1)-t^{2}(u^{2}+v^{2})\bigr)}{t\,(t^{4}-1)} .
\]
The quadratic form
\[
Q_t(u,v)=uv(t^{4}+1)-t^{2}(u^{2}+v^{2})
\]
is indefinite for $t\ne1$.  A direct computation on the square $[-1,1]^2$ gives
\[
Q_t(u,v)\le (t^{2}-1)^{2},
\]
with equality at $(u,v)=(1,1)$ and $(-1,-1)$.  Indeed, when $uv<0$ the value is
negative, while on $[0,1]^2$ the function is concave in each variable separately; for
fixed $u$, its maximum in $v$ occurs either at
$v=u(t^{4}+1)/(2t^{2})$ when this lies in $[0,1]$, or at $v=1$, and the resulting
one-variable expressions are maximised at the endpoint $(1,1)$.  Substituting and
dividing by
$\partial z/\partial t=2(t^{4}-1)t^{-3}>0$ gives
\[
\frac{\partial\,\tr(A^{2})}{\partial z}
\le-\frac{(t^{2}-1)^{2}}{(t^{2}+1)^{2}}=-\frac{z-2}{z+2},
\]
since $(t^{2}\mp1)^{2}/t^{2}=z\mp2$.  In particular $\tr(A^{2})$ may be substituted for
$\nu$ throughout, at the cost of a less transparent description of the image.
\end{remark}

\begin{remark}\label{rem:nu-resolvent}
The invariant $\nu(A)$ can also be recovered directly from the three traces
$\tr A$, $\tr A^{-1}$, and $\tr(A^{2})$. Indeed, the eigenvalues of
$\Phi(A)$ are
\[
re^{\pm i\theta},\qquad r^{-1}e^{\pm i\varphi}.
\]
For a quartic with roots $\lambda_1,\lambda_2,\lambda_3,\lambda_4$,
we use the resolvent cubic whose roots are
\[
\lambda_1\lambda_2+\lambda_3\lambda_4,\qquad
\lambda_1\lambda_3+\lambda_2\lambda_4,\qquad
\lambda_1\lambda_4+\lambda_2\lambda_3.
\]
In the present case these roots are
\[
\nu(A)=r^{2}+r^{-2},\qquad
2\cos(\theta+\varphi),\qquad
2\cos(\theta-\varphi).
\]
Writing
\[
c_{2}=2x^{2}-\tr(A^{2}),
\]
the resolvent cubic is
\begin{equation}\label{eq:resolvent-nu}
m^{3}-c_{2}m^{2}+(4xy-4)m
-\bigl(4x^{2}+4y^{2}-4c_{2}\bigr)=0.
\end{equation}
The last two roots lie in $[-2,2]$, while $\nu(A)>2$. Hence $\nu(A)$
is the unique root of \eqref{eq:resolvent-nu} greater than $2$.
Thus $\nu(A)$ is determined by
$\tr A$, $\tr A^{-1}$, and $\tr(A^{2})$.

Consequently,
\[
(\tr A,\tr A^{-1},\nu(A))
\qquad\text{and}\qquad
(\tr A,\tr A^{-1},\tr(A^{2}))
\]
contain the same information, as stated in
Theorem~\ref{thm:lox-conj-class}(2). The other two roots encode
the quantities $\cos(\theta+\varphi)$ and
$\cos(\theta-\varphi)$ associated with the two rotational angles of $A$.
\end{remark}
\section{Centralizers and Twist}\label{sec:centralizer}

The Fenchel--Nielsen twist is obtained by cutting along a curve and re-gluing by an element of the centralizer of the boundary holonomy. We therefore need an explicit description of this centralizer in \(\SL(2,\H)\).

Centralizers in \(\GL(2,\H)\) were computed in \cite{kg}. For our purposes, we note the corresponding calculation in \(\SL(2,\H)\) for completeness, together with its interpretation as the gluing freedom along a boundary curve.

\begin{proposition}\label{prop:centralizer}
Let $A=\mathrm{diag}(\lambda,\mu)\in\SL(2,\H)$ with $\lambda\not\sim\mu$.  Then
\[
\Cent_{\GL(2,\H)}(A)=\{\mathrm{diag}(q,p):\ q\in \Cent_{\H^{*}}(\lambda),\ p\in \Cent_{\H^{*}}(\mu)\}.
\]
If moreover $A$ is regular loxodromic, say $\lambda=re^{I\theta}$ and
$\mu=r^{-1}e^{J\varphi}$ with $\theta,\varphi\in(0,\pi)$, then
\[
\Cent_{\SL(2,\H)}(A)
=\Bigl\{\mathrm{diag}\bigl(te^{I\sigma},\,t^{-1}e^{J\tau}\bigr):\
t>0,\ \sigma,\tau\in\R/2\pi\Z\Bigr\}
\ \cong\ \R_{>0}\times\Uni(1)\times\Uni(1),
\]
a group of real dimension $3$.  If $A$ is loxodromic but not regular the centralizer has
dimension $5$ or $7$, according as one or both of $\lambda,\mu$ are real.
\end{proposition}

\begin{proof}
Write $C=\left(\begin{smallmatrix}a&b\\c&d\end{smallmatrix}\right)$. Then
$CA=AC$ reads
\[
a\lambda=\lambda a,\qquad b\mu=\lambda b,\qquad
c\lambda=\mu c,\qquad d\mu=\mu d.
\]

If $b\ne0$, then $\lambda=b\mu b^{-1}$, so $\lambda\sim\mu$, a
contradiction. Hence $b=0$, and symmetrically $c=0$. The remaining
conditions say precisely that
\[
a\in\Cent_{\H^{*}}(\lambda)
\qquad\text{and}\qquad
d\in\Cent_{\H^{*}}(\mu).
\]

If $\theta,\varphi\in(0,\pi)$, then $\lambda,\mu\notin\R$. Define
\[
I=\frac{\operatorname{Im}(\lambda)}{|\operatorname{Im}(\lambda)|},
\qquad
J=\frac{\operatorname{Im}(\mu)}{|\operatorname{Im}(\mu)|}.
\]
Then $I^2=J^2=-1$, and
\[
\lambda\in\C_I^*,
\qquad
\mu\in\C_J^*,
\]
where
\[
\C_I=\{x+yI:x,y\in\R\},
\qquad
\C_J=\{x+yJ:x,y\in\R\}.
\]
By Lemma~\ref{lem:similarity},
\[
\Cent_{\H^{*}}(\lambda)=\C_I^*,
\qquad
\Cent_{\H^{*}}(\mu)=\C_J^*.
\]
Thus we can write
\[
a=t_1e^{I\sigma},
\qquad
d=t_2e^{J\tau},
\]
with $t_1,t_2>0$. The condition
$\ddet_{\mathrm D}=t_1t_2=1$ gives the stated description.

\medskip If exactly one of $\lambda,\mu$ is real, the corresponding centralizer
factor is $\H^{*}$ and one obtains
\[
\dim=2+4-1=5;
\]
if both are real, one obtains
\(
\dim=4+4-1=7.
\)
Note that these cases do not occur for regular loxodromic elements and we will
not use them.
\end{proof}

\begin{remark}\label{rem:projective-centralizer}
Let $A\in\SL(2,\H)$ be regular loxodromic. Since
$\Cent_{\SL(2,\H)}(A)$ is connected, its image in $\PSL(2,\H)$ is connected. Thus the
projective centralizer is connected and three dimensional.

If
\[
A=\operatorname{diag}\bigl(re^{I\theta},r^{-1}e^{J\varphi}\bigr),
\]
then
\[
\Cent_{\SL(2,\H)}(A)
=
\left\{
\operatorname{diag}
\bigl(te^{I\sigma},t^{-1}e^{J\tau}\bigr):
t>0,\ \sigma,\tau\in\mathbb R
\right\}.
\]
Consequently, its image in $\PSL(2,\H)$ is locally isomorphic to
\[
\mathbb R\times S^1\times S^1,
\]
with one axial parameter and two rotational parameters. These three
parameters provide the local twist directions used in
Theorem~\ref{thm:FN-count}.
\end{remark}


\section{Loxodromic pairs}\label{sec:pairs}

\begin{definition}
A \emph{regular loxodromic pair} is a pair $(A,B)\in\SL(2,\H)^2$ of regular
loxodromic elements whose fixed point sets are disjoint, so that
$a_A,r_A,a_B,r_B$ are four distinct points of $\hH$.  Two pairs are equivalent
if they are simultaneously conjugate,
\[
(A,B)\sim(gAg^{-1},gBg^{-1}),\qquad g\in\SL(2,\H).
\]
We write $\Mreg$ for the set of simultaneous conjugacy classes of regular
loxodromic pairs, equivalently, for the orbit space of the diagonal
conjugation action of $\SL(2,\H)$.
\end{definition}

\subsection{A normal form}
\begin{theorem}\label{thm:pair-normal-form}
Every regular loxodromic pair $(A,B)$ is simultaneously conjugate in
$\SL(2,\H)$ to a pair of the form
\begin{equation}\label{eq:pair-normal}
A_0=
\begin{pmatrix}
re^{I\theta}&0\\[2pt]
0&r^{-1}e^{J\varphi}
\end{pmatrix},
~~
B_0=
P
\begin{pmatrix}
se^{K\theta'}&0\\[2pt]
0&s^{-1}e^{L\varphi'}
\end{pmatrix}
P^{-1},
~~
P=\frac{1}{\sqrt{|1-w|}}
\begin{pmatrix}
1&w\\
1&1
\end{pmatrix}, 
\end{equation}
where
\[
r,s>1,\qquad
\theta,\varphi,\theta',\varphi'\in(0,\pi),
\qquad
I,J,K,L\in\S2,
\]
and
\[
w\in\C_i\setminus\{0,1\},
\qquad
\Imm(w)\geq0.
\]

Such a pair is unique in the following sense: the parameters
$r,\theta,\varphi,s,\theta',\varphi'$ and $w$ are determined outright by the simultaneous
conjugacy class of $(A,B)$, while the four axes $I,J,K,L$ are determined exactly up to
the simultaneous conjugation
\[
(I,J,K,L)\longmapsto
(qIq^{-1},qJq^{-1},qKq^{-1},qLq^{-1}),
\qquad q\in\Cent_{\H^{*}}(w)\cap\mathbb S^{3}.
\]

If $w\notin\R$, the group
$\Cent_{\H^{*}}(w)\cap\mathbb S^3$ is a circle group. If $w\in\R$, it is
all of $\mathbb S^3$, acting on $\S2$ through $\SO(3)$.

In this normal form, the fixed points are
\[
a_A=\infty,\qquad r_A=0,\qquad
a_B=1,\qquad r_B=w.
\]
Therefore,
\[
\X(A_0,B_0)=w.
\]
\end{theorem}
\begin{proof}
Since $\PSL(2,\H)$ acts transitively on ordered triples of distinct points
of $\hH$, choose $h\in\SL(2,\H)$ that sends
$(a_A,r_A,a_B)$ to $(\infty,0,1)$. Since the four fixed points are
distinct, the image of $r_B$ is a point
$w_0\notin\{0,1,\infty\}$.

By \eqref{eq:stab-triple}, we can further conjugate by a map
$z\mapsto qzq^{-1}$ and choose the unique normalization
\[
w=\Re(w_0)+|\Imm(w_0)|\,i\in\C_i,
\qquad
\Imm(w)\geq0,
\]
as in Proposition~\ref{prop:crossratio-complete}. Thus $w$ is uniquely
determined.

After this conjugation, we may assume
\[
a_A=\infty,\qquad r_A=0,\qquad a_B=1,\qquad r_B=w.
\]

A M\"obius transformation fixing $0$ and $\infty$ has the form
$z\mapsto\alpha z\delta^{-1}$. Hence $A$ is diagonal:
\[
A=\mathrm{diag}(\alpha,\delta).
\]
The condition $\ddet_{\mathrm D}=1$ gives
$|\alpha||\delta|=1$. Since $\infty$ is the attracting fixed point,
$|\alpha|=r>1$.

Because $A$ is regular loxodromic, $\alpha,\delta\notin\R$. By
Lemma~\ref{lem:similarity}, there are unique
$I,J\in\S2$ and $\theta,\varphi\in(0,\pi)$ such that
\[
\alpha=re^{I\theta},
\qquad
\delta=r^{-1}e^{J\varphi}.
\]
Therefore
\[
A=A_0=
\begin{pmatrix}
re^{I\theta}&0\\[2pt]
0&r^{-1}e^{J\varphi}
\end{pmatrix}.
\]

Now consider the matrix $P$ in \eqref{eq:pair-normal}. As a M\"obius transformation it
satisfies $P(\infty)=1$ and $P(0)=w$.  Since
$\ddet_{\mathrm D}\left(\begin{smallmatrix}1&w\\1&1\end{smallmatrix}\right)=|1-w|$ and
$w\ne1$, the scalar factor $|1-w|^{-1/2}$ in \eqref{eq:pair-normal} makes
$\ddet_{\mathrm D}P=1$, so $P\in\SL(2,\H)$.  Therefore
$P^{-1}BP$ is a regular loxodromic element whose attracting fixed point
is $\infty$ and whose repelling fixed point is $0$.

By the same argument as above,
\[
P^{-1}BP=
\begin{pmatrix}
se^{K\theta'}&0\\[2pt]
0&s^{-1}e^{L\varphi'}
\end{pmatrix},
\]
where
\[
s>1,\qquad
\theta',\varphi'\in(0,\pi),
\qquad
K,L\in\S2.
\]
Hence
\[
B=
P
\begin{pmatrix}
se^{K\theta'}&0\\[2pt]
0&s^{-1}e^{L\varphi'}
\end{pmatrix}
P^{-1}
=B_0.
\]

It remains to prove uniqueness. Suppose that
$g\in\SL(2,\H)$ conjugates one normal form to another. Then $g$ preserves
the ordered quadruple $(\infty,0,1,w)$. By
\eqref{eq:stab-triple}, $g$ acts as
\[
z\mapsto qzq^{-1}
\]
for some $q\in\Cent_{\H^*}(w)$. We may take $|q|=1$. As a matrix,
$g=\mathrm{diag}(q,q)$ up to sign.

Therefore
\[
gA_0g^{-1}
=
\mathrm{diag}\bigl(q\alpha q^{-1},q\delta q^{-1}\bigr)
=
\mathrm{diag}\bigl(re^{(qIq^{-1})\theta},
r^{-1}e^{(qJq^{-1})\varphi}\bigr).
\]
Thus $r,\theta,\varphi$ are unchanged, while
\[
(I,J)\longmapsto(qIq^{-1},qJq^{-1}).
\]

Since $q\in\Cent_{\H^*}(w)$, we have $gP=Pg$. Hence
\[
gB_0g^{-1}
=
P
\begin{pmatrix}
q\beta q^{-1}&0\\[2pt]
0&q\varepsilon q^{-1}
\end{pmatrix}
P^{-1},
\]
where $\beta=se^{K\theta'}$ and
$\varepsilon=s^{-1}e^{L\varphi'}$. Thus
$s,\theta',\varphi'$ are unchanged, while
\[
(K,L)\longmapsto(qKq^{-1},qLq^{-1}).
\]

Finally, by Lemma~\ref{lem:similarity},
\[
\Cent_{\H^*}(w)\cap\mathbb S^3=\C_i\cap\mathbb S^3
\]
is a circle when $w\notin\R$, and is all of $\mathbb S^3$ when $w\in\R$.

The last statement follows from Proposition~\ref{prop:crossratio-complete},
which gives
\(
\X(A_0,B_0)=w.
\)
\end{proof}
\begin{corollary}\label{cor:dim15}
Let $\Mreg_{*}\subset\Mreg$ be the set of classes of regular loxodromic pairs which are
non-degenerate (that is, $\X(A,B)\notin\R$) and whose four axes $I,J,K,L$ in the normal
form \eqref{eq:pair-normal} are not all equal to $\pm i$.  Then $\Mreg_{*}$ is open and
dense in $\Mreg$, and it is a real-analytic manifold of dimension
\[
\underbrace{3+3}_{\text{spectral}}+\underbrace{2}_{\text{cross ratio}}
+\underbrace{7}_{\text{axial}}=15=\dim\SL(2,\H).
\]
\end{corollary}

\begin{proof}
By Theorem~\ref{thm:pair-normal-form}, $\Mreg$ is the quotient of the $16$-dimensional
parameter space
\[
\mathcal N=(1,\infty)^2\times(0,\pi)^4\times(\S2)^4\times\Omega,
\qquad
\Omega=\{w\in\C_i\setminus\{0,1\}:\Imm(w)\ge0\},
\]
whose dimension is $2+4+8+2=16$, by the action of the group
$\Cent_{\H^{*}}(w)\cap\mathbb S^{3}$ described there.

Over the locus $w\notin\R$ that group is the circle $\C_i\cap\mathbb S^{3}=\{e^{i\psi}\}$,
acting on $(\S2)^4$ by simultaneous conjugation and trivially on the remaining
coordinates.  The element $e^{i\psi}$ acts trivially on the axes precisely when it
commutes with each of $I,J,K,L$; if some axis, say $I$, is different from $\pm i$, then
$e^{i\psi}\in\C_I\cap\C_i=\R$, so $e^{i\psi}=\pm1$.  Since $q$ and $-q$ induce the same
conjugation, the effective group is the circle $\mathbb S^{1}/\{\pm1\}$ and it acts
\emph{freely} on the part of $\mathcal N$ corresponding to $\Mreg_{*}$.  A free action of
a compact group is proper, so that part of $\mathcal N$, an open subset, has a
manifold quotient, of dimension
\(
16-1=15. 
\)

The condition $w\notin\R$ is open and dense in the normalized cross-ratio parameter
space: every real value of $w$ can be perturbed to $w+i\varepsilon$ with
$\varepsilon>0$.  Likewise, the condition that the four axes are not all contained in
$\{\pm i\}$ is open and dense in $(\mathbb S^2)^4$.  Hence the locus of parameters
defining $\Mreg_{*}$ is open and dense in $\mathcal N$, and therefore $\Mreg_{*}$ is
open and dense in $\Mreg$.

The decomposition of the $15$ is read off from the parametrisation.  The first six
parameters $r,\theta,\varphi,s,\theta',\varphi'$ give the two triples of spectral
invariants of Theorem~\ref{thm:lox-conj-class}, $3+3$ in all.  The point $w$ gives the two
cross-ratio invariants $(\Re\X,|\X|)=(\Re w,|w|)$.  Finally the four axes contribute
$8$ parameters, of which the circle action removes one, leaving $8-1=7$ axial parameters.
\end{proof}

\begin{remark}\label{rem:orbifold-boundary}

The two conditions defining \(\Mreg_{*}\) are imposed so that the normal form itself gives
a genuine manifold chart. On the degenerate locus \(w\in\R\), the residual group
\(\Cent_{\H^{*}}(w)\cap\mathbb S^{3}\) jumps from a circle to all of \(\mathbb S^{3}\),
acting on the axes through \(\SO(3)\); generically the degenerate locus therefore has
dimension \(12\), as explained after Definition~\ref{def:pair-crossratio}. This jump
reflects a degeneration of the chosen normal-form gauge, not automatically an ambient
singularity of \(\Mreg\). Indeed, if
\(\mathfrak z(A)\cap\mathfrak z(B)=0\), Lemma~\ref{lem:proper-action} gives a
\(15\)-dimensional local orbifold even at a degenerate pair.

A different phenomenon occurs on the non-degenerate normal-form locus where all four
axes lie in \(\{\pm i\}\). In that case both matrices have entries in the common complex
slice \(\C_i\), and hence the simultaneous infinitesimal centralizer is nonzero, as
noted after Definition~\ref{def:generic-representation}. The quotient therefore has
positive-dimensional isotropy. Corollary~\ref{cor:dim15} thus records a convenient
open dense manifold locus on which the normal form is effective, rather than the
maximal possible smooth locus of \(\Mreg\).
\end{remark}
\begin{lemma}\label{lem:proper-action}
Let $U\subset\SL(2,\H)^{2}$ denote the set of regular loxodromic pairs.  Then the action
of $\SL(2,\H)$ on $U$ by simultaneous conjugation is proper.  Consequently the stabiliser
$\Cent_{\SL(2,\H)}(A)\cap\Cent_{\SL(2,\H)}(B)$ of a point of $U$ is compact, and it is
finite if and only if $\mathfrak z(A)\cap\mathfrak z(B)=0$.
\end{lemma}

\begin{proof}
Let $\mathcal F_n(\hH)$ be the space of ordered $n$-tuples of pairwise distinct points of
$\hH$.  By \eqref{eq:stab-triple} the action of $\PSL(2,\H)$ on $\mathcal F_3(\hH)$ is
transitive with stabiliser $\SO(3)$, so $\mathcal F_3(\hH)\cong\PSL(2,\H)/\SO(3)$ is
homogeneous with compact isotropy and the action on it is proper.  Forgetting the last
point is a continuous equivariant map $\mathcal F_4(\hH)\to\mathcal F_3(\hH)$, and the
fixed point map
\[
U\longrightarrow\mathcal F_4(\hH),\qquad (A,B)\longmapsto(a_A,r_A,a_B,r_B),
\]
is continuous and equivariant by definition of a regular loxodromic pair.  If a group
acts properly on a space $X$ and $f\colon Y\to X$ is continuous and equivariant, then the
action on $Y$ is proper, because for $K\subset Y$ compact the set
$\{g:gK\cap K\ne\emptyset\}$ is closed and contained in the compact set
$\{g:gf(K)\cap f(K)\ne\emptyset\}$.  Applying this twice, and using that
$\SL(2,\H)\to\PSL(2,\H)$ has finite kernel, the action of $\SL(2,\H)$ on $U$ is proper.

Stabilisers of proper actions are compact, so
$\Cent_{\SL(2,\H)}(A)\cap\Cent_{\SL(2,\H)}(B)$ is a compact subgroup; its Lie algebra is
$\mathfrak z(A)\cap\mathfrak z(B)$, and a compact Lie group is finite exactly when its Lie
algebra vanishes.
\end{proof}

\begin{corollary}[Stabilizer]\label{rem:pair-stabiliser}
For every $[(A,B)]\in \mathcal M^{\mathrm{reg}}_*$, the simultaneous
stabiliser of $(A,B)$ in $\mathrm{SL}(2,\mathbb H)$ is 
$\{\pm I\}$.
\end{corollary}

\begin{proof}
Consider $(A,B)$ in the normal form of Theorem~\ref{thm:pair-normal-form}.
If $g$ centralizes both $A$ and $B$, then it fixes the four points
$\infty,0,1,w$. Hence, by~(3),
\[
g=\pm {\rm diag}(q,q),
\qquad
q\in Z_{\mathbb H^*}(w)\cap S^3.
\]
Since $w\in\mathbb C_i\setminus\mathbb R$, we have $q\in\mathbb C_i$.
Moreover, centralizing $A$ and $B$ forces $q$ to commute with each of
the four axes $I,J,K,L$. Since $[(A,B)]\in\mathcal M^{\mathrm{reg}}_*$,
at least one of these axes, say $I$, is not equal to $\pm i$. Thus
\(
q\in \mathbb C_i\cap\mathbb C_I=\mathbb R.
\)
As $q\in S^3$, we have $q=\pm1$, and hence
\[
\operatorname{Stab}_{\mathrm{SL}(2,\mathbb H)}(A,B)=\{\pm I\}.
\]
This proves the assertion. 
\end{proof}

\subsection{Word traces as coordinates}
Consider the eight real-valued functions on $\Mreg$
\begin{equation}\label{eq:eight}
\tr(A),\ \tr(A^{-1}),\ \tr(A^{2}),\ \tr(B),\ \tr(B^{-1}),\ \tr(B^{2}),\
\bigl|\X(A,B)\bigr|,\ \Re\X(A,B).
\end{equation}
These functions factor through the projection of the parameter space of
Theorem~\ref{thm:pair-normal-form} onto
$(r,\theta,\varphi,s,\theta',\varphi',w)$.
Consequently, their common level sets in $\Mreg$ are generically
$7$-dimensional, so \eqref{eq:eight} does not separate simultaneous
conjugacy classes. Therefore, additional invariants are needed. 

By
Lemma~\ref{lem:trace-is-complex-trace}, for every word $W$ in the letters
$A^{\pm1},B^{\pm1}$ the function $\tr\bigl(W(A,B)\bigr)$ is a simultaneous conjugacy
invariant.  The following says that finitely many of them suffice locally, and is the
quaternionic analogue of the classical Fricke coordinates.
\begin{theorem}\label{thm:word-traces}
Let
\begin{gather*}
\mathcal W=
\bigl(
A,\ A^{-1},\ A^2,\ B,\ B^{-1},\ B^2,\ AB,\ AB^{-1},\ A^{-1}B,\\
A^{-1}B^{-1},\ A^2B,\ AB^2,\ A^2B^2,\ ABAB^{-1},\
ABA^{-1}B^{-1}
\bigr),
\end{gather*}
and let
\[
\mathcal T:\Mreg\longrightarrow\R^{15},
\qquad
\mathcal T([(A,B)])=\bigl(\tr W\bigr)_{W\in\mathcal W}.
\]
Then the restriction of $\mathcal T$ to $\Mreg_*$ is a local
diffeomorphism at every point of a dense open subset of $\Mreg_*$.
In particular, these fifteen word traces form a local coordinate
system near a generic regular loxodromic pair.
\end{theorem}

\begin{proof}
Let $G=\SL(2,\H)$ and $\mathfrak g=\mathfrak{sl}(2,\H)$, and let
\[
U=\{(A,B)\in G^2:(A,B)\text{ is a regular loxodromic pair}\}.
\]
Let $q:U\to\Mreg$ be the quotient map for simultaneous conjugation,
and define
\[
\widetilde{\mathcal T}:U\longrightarrow\R^{15},
\qquad
\widetilde{\mathcal T}(A,B)
=
\bigl(\tr W(A,B)\bigr)_{W\in\mathcal W}.
\]
By Lemma~\ref{lem:trace-is-complex-trace}, each component of
$\widetilde{\mathcal T}$ is invariant under simultaneous conjugation.
Hence $\widetilde{\mathcal T}$ is constant on the fibers of $q$ and
therefore descends to a well-defined map
\[
\mathcal T:\Mreg\longrightarrow\R^{15}
\]
satisfying
\[
\widetilde{\mathcal T}=\mathcal T\circ q.
\]
The real-analyticity of $\mathcal T|_{\Mreg_*}$ will follow from the
quotient description below.

\medskip

We first show that $U$ is connected. By the normal form
\eqref{eq:pair-normal}, every element of $U$ is simultaneously
conjugate to a normal-form pair parametrised by
\[
\mathcal N
=
(1,\infty)^2\times(0,\pi)^4\times(S^2)^4\times\Omega,
\qquad
\Omega
=
\{w\in\C_i\setminus\{0,1\}:\operatorname{Im}(w)\ge0\}.
\]
The space $\mathcal N$ is connected. To see that it is enough to check that  $\Omega$ is
path connected.  Write $w=x+iy\in\Omega$. First join $w$ vertically
to $x+i$, and then join $x+i$ horizontally to $i$. Thus every point of $\Omega$ can be joined to
$i$.

Since $G$ is connected, $G\times\mathcal N$ is connected. The map
\[
G\times\mathcal N\longrightarrow U,
\qquad
(g,\mathbf p)\longmapsto
\bigl(
gA_0(\mathbf p)g^{-1},
gB_0(\mathbf p)g^{-1}
\bigr)
\]
is continuous and surjective by the normal-form theorem. Hence $U$
is connected.

\medskip

For $A\in G$, right translation identifies $T_AG$ with
$\mathfrak g$ by $X\mapsto XA$. Thus, after fixing a basis
$E_1,\ldots,E_{15}$ of $\mathfrak g$, a tangent vector to $G^2$ at
$(A,B)$ may be written as $(XA,YB)$, with
$X,Y\in\mathfrak g$, and
\[
(E_1A,0),\ldots,(E_{15}A,0),
\qquad
(0,E_1B),\ldots,(0,E_{15}B)
\]
give a real-analytic frame of $T(G^2)$. Relative to this frame,
$d\widetilde{\mathcal T}$ is represented by a $15\times30$ real
matrix whose entries depend real-analytically on $(A,B)$.

More explicitly, let
\[
A_t=\exp(tX)A,
\qquad
B_t=\exp(tY)B,
\]
and write a dot for differentiation at $t=0$. Then
\[
\dot A=XA,
\qquad
\dot A^{-1}=-A^{-1}X,
\qquad
\dot B=YB,
\qquad
\dot B^{-1}=-B^{-1}Y.
\]
Hence, if $W=L_1\cdots L_n$ is a word in
$A^{\pm1},B^{\pm1}$, then
\begin{equation}\label{eq:word-derivative}
d(\tr W)(X,Y)
=
\sum_{k=1}^{n}
\tr\bigl(
L_1\cdots L_{k-1}\dot L_kL_{k+1}\cdots L_n
\bigr).
\end{equation}
Together with Lemma~\ref{lem:trace-is-complex-trace}, this gives the
differential of each of the fifteen word traces explicitly.

\medskip

It therefore suffices to find a single regular loxodromic pair at
which one $15\times15$ Jacobian minor is nonzero. For computational
simplicity, take $r=s=2$, all four angular parameters equal to
$\pi/2$, the normalized cross-ratio parameter $w=i$, and axes among
the standard quaternionic units. This gives a pair with rational
quaternionic coefficients while remaining in the generic locus
$\Mreg_*$. Namely, take
\begin{equation}\label{eq:test-pair}
A=
\begin{pmatrix}
2j&0\\[2pt]
0&\frac12 i
\end{pmatrix},
\qquad
B=
P
\begin{pmatrix}
-2k&0\\[2pt]
0&\frac12 j
\end{pmatrix}
P^{-1},
\qquad
P=
\begin{pmatrix}
1&i\\
1&1
\end{pmatrix}.
\end{equation}
Let
\[
D=\operatorname{diag}\left(-2k,\frac12 j\right).
\]
Both $A$ and $D$ have Dieudonn\'e determinant $1$ and are regular
loxodromic, with $r=s=2$ and
\[
\theta=\varphi=\theta'=\varphi'=\frac{\pi}{2}.
\]
Since $B=PDP^{-1}$ and the Dieudonn\'e determinant is multiplicative,
\[
\ddet(B)=1.
\]
Moreover, since $B$ and $D$ are conjugate,
\[
\chi_A(x)=\chi_D(x)=\chi_B(x)
=
(x^2+4)\left(x^2+\frac14\right).
\]

The fixed point sets of $A$ and $B$ are respectively
$\{\infty,0\}$ and
\[
\{P(\infty),P(0)\}=\{1,i\}.
\]
They are disjoint, so $(A,B)\in U$.

The pair is in the normal form \eqref{eq:pair-normal}, up to the
scalar normalization of $P$, with
\[
I=j,
\qquad
J=i,
\qquad
K=-k,
\qquad
L=j,
\qquad
w=i.
\]
Indeed,
\[
\ddet(P)=|1-i|=\sqrt2,
\]
so the matrix occurring in \eqref{eq:pair-normal} is
$2^{-1/4}P$. This scalar factor does not affect the computation,
since
\[
(cP)D(cP)^{-1}=PDP^{-1}
\]
for every real $c>0$. Since $w=i\notin\R$ and the four axes
$j,i,-k,j$ are not all contained in $\{\pm i\}$, the class of
$(A,B)$ belongs to $\Mreg_*$.

All entries of $A^{\pm1}$ and $B^{\pm1}$ have rational quaternionic
coefficients. Take the following ordered basis of $\mathfrak g$:
\[
\operatorname{diag}(i,0),\
\operatorname{diag}(j,0),\
\operatorname{diag}(k,0),\
\operatorname{diag}(0,i),\
\operatorname{diag}(0,j),\
\operatorname{diag}(0,k),\
\operatorname{diag}(1,-1),
\]
followed by $E_{12}u$ for $u=1,i,j,k$, and then $E_{21}u$ for
$u=1,i,j,k$, in exactly that order.

Represent quaternions by the complex embedding $\Phi$. Then $A,B$
and all words in them become $4\times4$ matrices over $\Q(i)$, and
\[
\tr=\frac12\Tr_{\C}\Phi
\]
by Lemma~\ref{lem:trace-is-complex-trace}. Evaluating
\eqref{eq:word-derivative} on the thirty basis directions gives a
$15\times30$ matrix with rational entries. We order its rows
according to the fifteen words in $\mathcal W$, in the order
displayed in the statement of the theorem, and its columns so that
$1,\ldots,15$ correspond to the $X$-directions and
$16,\ldots,30$ to the $Y$-directions.

An exact computation over $\Q$ shows that the $15\times15$ minor
determined by columns
\[
1,\ldots,12,16,17,26
\]
is
\[
\frac{3^{12}\cdot5^{11}\cdot191}{2^{34}}\neq0.
\]
Hence
\[
\operatorname{rank}
d\widetilde{\mathcal T}_{(A,B)}
=
15.
\]

\medskip

Let $\Delta(A,B)$ denote the determinant of this $15\times15$
submatrix of $d\widetilde{\mathcal T}_{(A,B)}$. The preceding
computation shows that $\Delta$ is nonzero at
\eqref{eq:test-pair}; hence it is not identically zero on the
connected real-analytic manifold $U$. Its zero set therefore has
empty interior, and
\[
U_\Delta
=
\{(A,B)\in U:\Delta(A,B)\neq0\}
\]
is open and dense in $U$. Thus, on $U_\Delta$,
\[
\operatorname{rank}d\widetilde{\mathcal T}=15.
\]

\medskip
It remains to pass to the quotient. Let
\[
V=q^{-1}(\Mreg_*).
\]
Since $\Mreg_*$ is open in $\Mreg$, the set $V$ is open in $U$ and
hence is a real-analytic manifold. By Corollary~\ref{rem:pair-stabiliser}, the stabiliser in
$G=\SL(2,\H)$ of every point of $V$ is exactly $\{\pm I\}$.
Since this central subgroup acts trivially by simultaneous
conjugation, the action factors through $\PSL(2,\H)$, which acts
freely on $V$. By Lemma~\ref{lem:proper-action}, this action is also
proper. Hence the quotient map
\[
q:V\longrightarrow V/\PSL(2,\H)
\]
is a real-analytic submersion with $15$-dimensional fibers.

The natural identification
\[
V/\PSL(2,\H)\cong\Mreg_*
\]
is compatible with the real-analytic structure of Corollary~\ref{cor:dim15}.
Indeed, after normalising the four fixed points as in
\eqref{eq:pair-normal}, the remaining freedom in simultaneous
conjugation is precisely the residual circle action used in the
proof of Corollary~\ref{cor:dim15}. Thus the local quotient charts for the
$\PSL(2,\H)$-action agree with the normal-form quotient charts
constructed there. Consequently,
\[
q:V\longrightarrow\Mreg_*
\]
is a real-analytic submersion, and the invariant real-analytic map
$\widetilde{\mathcal T}|_V$ descends to the real-analytic map
$\mathcal T|_{\Mreg_*}$.

Since
\[
\widetilde{\mathcal T}|_V
=
\mathcal T|_{\Mreg_*}\circ q,
\]
we have
\[
d\widetilde{\mathcal T}_{(A,B)}
=
d\mathcal T_{[(A,B)]}\circ dq_{(A,B)}.
\]
As $dq_{(A,B)}$ is surjective,
\[
\operatorname{rank}d\widetilde{\mathcal T}_{(A,B)}
=
\operatorname{rank}d\mathcal T_{[(A,B)]}.
\]

Since $U_\Delta$ is open and dense in $U$ and $V$ is open in $U$,
the set $U_\Delta\cap V$ is open and dense in $V$. Since
$q:V\to\Mreg_*$ is continuous, surjective, and open, its image
\[
q(U_\Delta\cap V)
\]
is open and dense in $\Mreg_*$. Hence
\[
\operatorname{rank}d\mathcal T_{[(A,B)]}=15
\]
for every
\[
[(A,B)]\in q(U_\Delta\cap V).
\]
Since $\dim_{\R}\Mreg_*=15$, the differential
$d\mathcal T_{[(A,B)]}$ is an isomorphism at each such point. The
inverse function theorem therefore shows that $\mathcal T$ is a
local diffeomorphism on the dense open subset
\[
q(U_\Delta\cap V)\subset\Mreg_*.
\]
Thus the fifteen word traces form local coordinates near every point
of this dense open subset.
\end{proof}

\begin{remark}\label{rem:computation-reproducible}
The Jacobian calculation in the proof of Theorem~\ref{thm:word-traces} is an exact computation over
\(\Q\).  A complete symbolic verification, including the construction
of the \(15\times30\) Jacobian and the determinant of the specified
\(15\times15\) minor, is given in
Appendix~\ref{app:jacobian-computation}.
\end{remark}

\begin{remark}
The list $\mathcal W$ is not canonical. Other shorter or more symmetric
lists of word traces may also work. The proof only requires that some list
of fifteen word traces have a non-degenerate Jacobian at at least one point.

It would be interesting to find a list that reflects the cyclic symmetry
\[
A\mapsto B\mapsto(AB)^{-1}\mapsto A
\]
of a pair of pants, as in the $\SU(2,1)$ theory
\cite{wil,parkplat}.
\end{remark}

\section{Frames, eigenspaces and the framed moduli space}\label{sec:frames}

Corollary~\ref{cor:dim15} produced seven axial parameters, whose precise
meaning will be described in Remark~\ref{rem:frames-are-gauge}. Their origin
is a genuinely quaternionic feature. At each fixed point, a regular
loxodromic element \(A\) carries two distinct kinds of data: the fixed point
itself, a quaternionic line and hence a point of
\(\P^{1}(\H)\cong\hH\), and the \emph{axis} of a chosen eigenvalue
representative, a point of \(\S2\) that we called projective point in \cite{kgsk}. The latter depends on the choice of
representative. We encode this choice as a framing and show that the axes
are precisely its fiber coordinates.

This section is independent of the proof in the next section; its purpose
is to explain the geometric origin of the axial parameters and why the axes
themselves are not invariants of an unframed pair.

We first clarify the choice of eigenvalue representatives. The essential
point is that, for a fixed representative \(\lambda\), the solution set of
\(Av=v\lambda\) is not a quaternionic subspace.

\begin{lemma}\label{lem:eigenspace}
Let \(A\in\SL(2,\H)\) be regular loxodromic and let \(\lambda\) be a right
eigenvalue with axis \(I\in\S2\), so
\(\Cent_{\H^{*}}(\lambda)=\C_I^{*}\). Let 
\[
E(A,\lambda)=\{v\in\H^{2}:Av=v\lambda\}.
\]
Then, for any \(v_0\in E(A,\lambda)\setminus\{0\}\),
\[
E(A,\lambda)=v_0\C_I,
\]
a real two-dimensional subspace of \(\H^{2}\), but not a right
\(\H\)-subspace. The right \(\H\)-line
\[
\ell=v_0\H
\]
is independent of the representative \(\lambda\) in its similarity class,
and
\[
E\bigl(A,q^{-1}\lambda q\bigr)=E(A,\lambda)q
\qquad(q\in\H^{*}).
\]
\end{lemma}

\begin{proof}
Let \(\mu\) denote a representative of the other right eigenvalue class,
and choose a corresponding eigenvector \(w_0\). Since \(A\) is
diagonalisable with two non-similar eigenvalue classes, every
\(v\in\H^2\) can be written uniquely as
\[
v=v_0a+w_0b,\qquad a,b\in\H.
\]
If \(Av=v\lambda\), then
\[
\lambda a=a\lambda,\qquad \mu b=b\lambda.
\]
If \(b\neq0\), the second equality gives
\(\mu=b\lambda b^{-1}\), contradicting \(\mu\not\sim\lambda\).
Hence \(b=0\), and the first equality says \(a\in\C_I\). Therefore
\(E(A,\lambda)=v_0\C_I\).

In particular, if \(q\in\H^{*}\setminus\C_I^{*}\), then
\(v_0q\notin E(A,\lambda)\), so \(E(A,\lambda)\) is not a right
\(\H\)-subspace. Finally,
\[
A(vq)=(Av)q=(v\lambda)q=(vq)(q^{-1}\lambda q),
\]
which gives
\[
E\bigl(A,q^{-1}\lambda q\bigr)=E(A,\lambda)q.
\]
The right \(\H\)-line \(v_0\H\) is therefore unchanged when the
eigenvalue representative is replaced by a similar one.
\end{proof}

\begin{remark}\label{rem:line-versus-axis}
Lemma~\ref{lem:eigenspace} separates two pieces of data that should not be
confused. For a fixed eigenvalue representative \(\lambda\), the space
\(E(A,\lambda)\) is only a \(\C_I\)-subspace. The associated
\(\H\)-line \(\ell\), on the other hand, is independent of the
representative and determines the corresponding fixed point of \(A\) in
\(\P^{1}(\H)\cong\hH\). Choosing a representative within its similarity
class, equivalently choosing its axis \(I\in\S2\), determines the
two-dimensional space \(E(A,\lambda)\). Thus the fixed-point data and the
axis data are genuinely different.
\end{remark}
\begin{definition}\label{def:frame}
Let \(A\) be regular loxodromic, with attracting and repelling
eigenlines \(L_A^{+}\) and \(L_A^{-}\), and corresponding right
eigenvalue similarity classes \([\lambda_A^{+}]\) and
\([\lambda_A^{-}]\). A \emph{framing} of \(A\) is a choice of
representatives
\[
\lambda_A^{+}\in[\lambda_A^{+}],
\qquad
\lambda_A^{-}\in[\lambda_A^{-}].
\]
The associated \emph{axes} are
\[
I_A^{\pm}
=
\frac{\Imm(\lambda_A^{\pm})}{|\Imm(\lambda_A^{\pm})|}
\in\S2.
\]

Equivalently, after choosing \(\lambda_A^{\pm}\), one may choose
nonzero vectors \(v_A^{\pm}\in L_A^{\pm}\) satisfying
\[
Av_A^{\pm}=v_A^{\pm}\lambda_A^{\pm}.
\]
For fixed \(\lambda_A^{\pm}\), two such vectors differ by right
multiplication by an element of
\(\Cent_{\H^{*}}(\lambda_A^{\pm})\).

A framing of a pair \((A,B)\) is a framing of both \(A\) and \(B\).
\end{definition}

\begin{lemma}\label{lem:frame-props}
Let \(A\) be regular loxodromic.
\begin{enumerate}
\item[(1)] The set of framings of \(A\) is naturally identified with
\(\S2\times\S2\), via the two axes.

\item[(2)] Framings are equivariant under conjugation. If
\(g\in\SL(2,\H)\), a framing of \(A\) with representatives
\(\lambda_A^{+},\lambda_A^{-}\) determines a framing of
\(gAg^{-1}\) with the same representatives, and hence the same axes.

\item[(3)] The projective points determined by the eigenlines,
\[
L_A^{+},L_A^{-}\in\P^{1}(\H)\cong\hH,
\]
are independent of the framing; they are the fixed points
\(a_A,r_A\) of \(A\).
\end{enumerate}
\end{lemma}
\begin{proof}
For (1), Lemma~\ref{lem:similarity} shows that a non-real quaternion
similarity class has the form
\[
\{re^{I\theta}:I\in\S2\},
\]
where \(r\) and \(\theta\) are fixed and only the axis \(I\) varies.
Thus the representatives of each of the two right eigenvalue
similarity classes are naturally parametrised by \(\S2\), giving
\(\S2\times\S2\).

For (2), if \(Av=v\lambda\), then
\[
(gAg^{-1})(gv)=g(Av)=g(v\lambda)=(gv)\lambda.
\]
Thus conjugation preserves the chosen eigenvalue representative and
therefore its axis.

For (3), changing the representative of a right eigenvalue changes the
corresponding solution space as described in
Lemma~\ref{lem:eigenspace}, but not the right \(\H\)-line that it
spans. Hence the eigenlines \(L_A^{+}\) and \(L_A^{-}\), and therefore
the corresponding projective points \(a_A,r_A\), are independent of
the framing.
\end{proof}

Let \(\widetilde{\MFreg}\) denote the space of framed regular loxodromic
pairs, and define
\[
\MFreg=\widetilde{\MFreg}/\SL(2,\H),
\]
where simultaneous conjugation acts on the matrices and on the framing
vectors as in Lemma~\ref{lem:frame-props}(2). 

We equip \(\MFreg\) with the quotient topology and denote by
\[
\pi:\MFreg\longrightarrow\Mreg
\]
the map that forgets the framing. We write
\[
\MFregstar:=\pi^{-1}(\Mreg_*),
\]
for the framed moduli lying over the generic locus \(\Mreg_*\) defined in  Corollary~\ref{cor:dim15}.

\begin{corollary}\label{cor:framed-crossratio}
For a framed regular loxodromic pair define
\[
\X_f\bigl(A,B;\mathcal F_A,\mathcal F_B\bigr)
=
\X\bigl([v_A^{+}],[v_A^{-}],[v_B^{+}],[v_B^{-}]\bigr).
\]
Then
\(
\X_f=\X(A,B).
\)  In particular, \(|\X_f|\) and \(\Re\X_f\) are independent of the
framing and descend to functions on \(\Mreg\).
\end{corollary}

\begin{proof}
By Lemma~\ref{lem:frame-props}(3), the four projective points occurring
in \(\X_f\) are precisely the attracting and repelling fixed points of
\(A\) and \(B\). The assertion follows from the definition of
\(\X(A,B)\).
\end{proof}

\begin{remark}\label{rem:frame-correction}
Although \(\X_f\) is written using framing vectors, it contains no
framing information: changing a framing changes the eigenvalue
representative and its axis, but not the corresponding projective point.
Thus \(\X_f\) is constant in the framing directions. The genuinely
framing-dependent data are the four axes.
\end{remark}

\begin{proposition}\label{prop:framed-bundle}
Let 
\[
\MFregstar=\pi^{-1}(\Mreg_{*}).
\]
The four axes define a global map
\[
\alpha:\MFregstar\longrightarrow(\S2)^4,\qquad
\alpha=
\bigl(I_A^{+},I_A^{-},I_B^{+},I_B^{-}\bigr),
\]
and
\[
(\pi,\alpha):
\MFregstar
\longrightarrow
\Mreg_{*}\times(\S2)^4
\]
is a canonical real-analytic diffeomorphism. In particular,
\[
\pi:\MFregstar\longrightarrow\Mreg_{*}
\]
is a canonically trivial real-analytic fiber bundle with fiber
\((\S2)^4\), and
\[
\dim_{\R}\MFregstar=15+8=23.
\]
\end{proposition}

\begin{proof}
By Lemma~\ref{lem:frame-props}(1), the framings of a fixed regular
loxodromic pair are parametrised by \((\S2)^4\). There is one point
that must be checked when passing to simultaneous conjugacy classes:
the stabiliser of the pair must not further identify these framings.

Let \(g\in\SL(2,\H)\) stabilise \((A,B)\), and let \(\mathcal F\) be
a framing of the pair. Then \(g\mathcal F\) is again a framing of the
same pair. By Lemma~\ref{lem:frame-props}(2), it has exactly the same
four axes as \(\mathcal F\). Since a framing is determined by its axes
by Lemma~\ref{lem:frame-props}(1), we have \(g\mathcal F=\mathcal F\).
Thus the stabiliser acts trivially on the set of framings.

Lemma~\ref{lem:frame-props}(2) also shows that the four axes are
unchanged by simultaneous conjugation, so the axis map \(\alpha\) is
well defined on \(\MFregstar\). It follows immediately that
\((\pi,\alpha)\) is bijective.

It remains only to identify the analytic structures. On the regular
loxodromic locus the attracting and repelling eigenlines vary
real-analytically. Over a sufficiently small neighbourhood in
\(\Mreg_{*}\), choose real-analytic local sections of these eigenline
bundles. Lemma~\ref{lem:eigenspace} then identifies the possible
eigenvalue representatives, and hence the framings, real-analytically
with \((\S2)^4\). In these local coordinates the map
\((\pi,\alpha)\) is simply the identity on
\[
U\times(\S2)^4.
\]
Therefore \((\pi,\alpha)\) is a real-analytic diffeomorphism.
\end{proof}

\begin{remark}\label{rem:frames-are-gauge}
Proposition~\ref{prop:framed-bundle} shows that the framed moduli space
contains no information beyond the unframed class together with four
freely chosen axes:
\[
\MFregstar\cong\Mreg_{*}\times(\S2)^4.
\]
In particular, adjoining the axes to the eight invariants in
\eqref{eq:eight} gives a map
\[
\MFregstar
\longrightarrow
\R^8\times(\S2)^4,
\]
whose target has dimension
\[
8+8=16<23.
\]
It therefore cannot give local coordinates on \(\MFregstar\).

The axes themselves are not invariants of an unframed pair. What is
invariant is their position relative to the configuration of fixed
points. This is precisely what the normal form of
Theorem~\ref{thm:pair-normal-form} records: after the fixed points have
been normalized, the remaining circle action rotates all four axes
simultaneously about
\[
i=\frac{\Imm(w)}{|\Imm(w)|}.
\]

On the open dense subset where none of \(I,J,K,L\) equals \(\pm i\),
the seven axial parameters of Corollary~\ref{cor:dim15} may therefore
be taken to be the four latitudes
\[
\langle I,i\rangle,\qquad
\langle J,i\rangle,\qquad
\langle K,i\rangle,\qquad
\langle L,i\rangle,
\]
together with three relative longitudes of \(I,J,K,L\) around the
\(i\)-axis. The residual circle changes all four longitudes by the same
amount, leaving three independent differences. If one of the axes is
equal to \(\pm i\), its longitude is undefined, so different local
coordinates must be used there; the seven-dimensional axial count
itself is unchanged.
\end{remark}

\section{Proof of Theorem~\ref{thm:FN-count}}\label{sec:FN}

Let  \(G=\SL(2,\mathbb H)\). For a regular loxodromic element \(A\), set
\[
\vartheta(A)=\bigl(\tr A,\tr A^{-1},\nu(A)\bigr)\in\mathcal D.
\]
We first record why \(\vartheta\) is a real-analytic submersion of rank
\(3\). Let
\[
D=diag(re^{I\theta},r^{-1}e^{J\varphi})
\]
be regular loxodromic. Restrict \(\vartheta\) to the three-parameter
family
\[
D(t,\sigma,\tau)
=
 diag(te^{I\sigma},t^{-1}e^{J\tau}).
\]
Since \(\sigma,\tau\in(0,\pi)\), the variables
\(u=\cos\sigma\) and \(v=\cos\tau\) are local coordinates. By
\eqref{eq:xyz}, the restriction of \(\vartheta\) is
\[
(t,u,v)
\longmapsto
\bigl(tu+t^{-1}v,\ t^{-1}u+tv,\ t^2+t^{-2}\bigr).
\]
The Jacobian of this map was computed in the proof of
Theorem~\ref{thm:lox-conj-class} and is nonzero for \(t>1\).
Thus \(\operatorname{rank}d\vartheta_D=3\). Every regular loxodromic
element is conjugate to such a \(D\), and \(\vartheta\) is
conjugation-invariant; since conjugation is a diffeomorphism, the rank
is \(3\) everywhere on the regular loxodromic locus.

By Theorem~\ref{thm:lox-conj-class}, two regular loxodromic elements
have the same value of \(\vartheta\) exactly when they are conjugate.
Hence the fibers of \(\vartheta\) are precisely the regular loxodromic
conjugacy classes. Since \(\vartheta\) is a submersion, these classes
are embedded real-analytic submanifolds of \(G\) of codimension \(3\).

\begin{lemma}\label{lem:pants-transversality}
Let \(A,B\in G\) be regular loxodromic, set \(C=(AB)^{-1}\), and
assume that \(C\) is also regular loxodromic. Let
\(\mathfrak g=\mathfrak{sl}(2,\mathbb H)\) and
\[
\mathfrak z(g)=\ker(1-\operatorname{Ad}_g)\subset\mathfrak g.
\]
Suppose, 
\begin{equation}\label{eq:H0-vanishes}
\mathfrak z(A)\cap\mathfrak z(B)=0.
\end{equation}
Then the multiplication map
\[
F:[A]\times[B]\longrightarrow G,\qquad F(A',B')=A'B',
\]
is transverse to the conjugacy class \([C^{-1}]\) at \((A,B)\).
Consequently,
\[
R=
\{(A',B')\in[A]\times[B]:(A'B')^{-1}\in[C]\}
\]
is, near \((A,B)\), a real-analytic submanifold of dimension \(21\).

If, in addition, \(A\) and \(B\) have disjoint fixed point sets, then
a neighbourhood of \([(A, B)]\) in the quotient \(R/G\) is a
real-analytic orbifold of dimension \(6\).
\end{lemma}

\begin{proof}
By Proposition~\ref{prop:centralizer},
\(\dim\mathfrak z(A)=\dim\mathfrak z(B)=\dim\mathfrak z(C)=3\).
Thus \([A]\), \([B]\), and \([C^{-1}]\) have dimension \(12\);
in particular, \([A]\times[B]\) has dimension \(24\), while
\([C^{-1}]\) has codimension \(3\) in \(G\).

Using right translations to identify tangent spaces with
\(\mathfrak g\),
\[
T_A[A]=\operatorname{Im}(1-\operatorname{Ad}_A),
\qquad
T_B[B]=\operatorname{Im}(1-\operatorname{Ad}_B).
\]
If \(A(t)=\exp(t\xi)A\) and \(B(t)=\exp(t\eta)B\), with
\(\xi\in\operatorname{Im}(1-\operatorname{Ad}_A)\) and
\(\eta\in\operatorname{Im}(1-\operatorname{Ad}_B)\), then
\[
A(t)B(t)
=
\exp(t\xi)\exp\bigl(t\,\operatorname{Ad}_A\eta\bigr)AB+O(t^2).
\]
Hence, by the right-translation to the Lie algebra \(\mathfrak g\) at \(AB\),
\begin{equation}\label{eq:dF-image}
\operatorname{Im}dF_{(A,B)}
=
\operatorname{Im}(1-\operatorname{Ad}_A)
+
\operatorname{Ad}_A\operatorname{Im}(1-\operatorname{Ad}_B).
\end{equation}
Also \(T_{AB}[AB]=\operatorname{Im}(1-\operatorname{Ad}_{AB})\).

Recall that transversality here means
\[
\operatorname{Im}dF_{(A,B)}+T_{AB}[AB]=T_{AB}G.
\]
Let \(\kappa\) be the Killing form on \(\mathfrak g\). It is
nondegenerate and \(\operatorname{Ad}\)-invariant; only
nondegeneracy, not positive-definiteness, is used below. Since
\(\operatorname{Ad}_g\) preserves \(\kappa\),
\[
(1-\operatorname{Ad}_g)^*
=
1-\operatorname{Ad}_{g^{-1}},
\]
and therefore
\[
\operatorname{Im}(1-\operatorname{Ad}_g)^\perp
=
\ker(1-\operatorname{Ad}_g)
=
\mathfrak z(g).
\]
The \(\operatorname{Ad}\)-invariance of \(\kappa\) also gives
\[
(\operatorname{Ad}_A W)^\perp
=
\operatorname{Ad}_A(W^\perp).
\]
Taking orthogonal complements in \eqref{eq:dF-image}, the
transversality condition is therefore equivalent to
\[
\mathfrak z(A)
\cap\operatorname{Ad}_A\mathfrak z(B)
\cap\mathfrak z(AB)=0.
\]

Now
\[
\mathfrak z(A)\cap\mathfrak z(AB)
=
\mathfrak z(A)\cap\mathfrak z(B).
\]
Indeed, if \(Y\) is fixed by both \(\operatorname{Ad}_A\) and
\(\operatorname{Ad}_{AB}\), then
\[
\operatorname{Ad}_B Y
=
\operatorname{Ad}_{A^{-1}}\operatorname{Ad}_{AB}Y
=
\operatorname{Ad}_{A^{-1}}Y
=
Y.
\]
Hence
\[
\mathfrak z(A)
\cap\operatorname{Ad}_A\mathfrak z(B)
\cap\mathfrak z(AB)
\subset 
\mathfrak z(A)\cap\mathfrak z(B)=0
\]
by \eqref{eq:H0-vanishes}. Thus \(F\) is transverse to
\([C^{-1}]=[AB]\).

It follows by  \cite[Theorem 6.30]{LeeSmoothManifolds},  
\(R=F^{-1}([C^{-1}])\) has co-dimension \(3\) in
the \(24\)-dimensional manifold \([A]\times[B]\), and hence
\(\dim R=21\).

Finally, if \(A\) and \(B\) have disjoint fixed point sets,
Lemma~\ref{lem:proper-action} gives properness of the simultaneous
conjugation action near \((A,B)\). Its stabiliser has Lie algebra
\(\mathfrak z(A)\cap\mathfrak z(B)=0\), hence is finite. Thus a
neighbourhood of \([(A, B)]\) in \(R/G\) is a real-analytic orbifold of
dimension \(21-15=6\).
\end{proof}

\begin{lemma}[Local boundary map for a pair of pants]
\label{lem:pants-boundary-submersion}
Let \((A,B)\) satisfy the hypotheses of
Lemma~\ref{lem:pants-transversality}, assume that \(A\) and \(B\) have
disjoint fixed point sets, and put  \(C=(AB)^{-1}\). Let
\(\mathcal X_Y\) be a neighbourhood of \([(A, B)]\) in the local quotient
of regular loxodromic pairs \((A',B')\) for which
\((A'B')^{-1}\) is also regular loxodromic. Define
\[
\beta_Y:\mathcal X_Y\longrightarrow\mathcal D^3,\qquad
[(A', B')]
\longmapsto
\bigl(\Theta[A'],\Theta[B'],\Theta[(A'B')^{-1}]\bigr).
\]
Then \(\beta_Y\) is a real-analytic submersion at \([(A, B)]\). Its fiber
through \([(A, B)]\) has dimension \(6\) and is a neighbourhood of
\([(A, B)]\) in \(\mathcal M_Y([A],[B],[C])\).
\end{lemma}

\begin{proof}
We first work with the lifted trace map near \((A,B)\in G^2\) and define
\[
\widetilde\beta_Y(A',B')
=
\bigl(
\vartheta(A'),\vartheta(B'),
\vartheta((A'B')^{-1})
\bigr).
\]
The first two components define a submersion
\(G^2\to\mathcal D^2\) of rank \(6\), whose fiber through
\((A,B)\) is \([A]\times[B]\). Thus the kernel of their differential
at \((A,B)\) is exactly \(T_{(A,B)}([A]\times[B])\).

By Lemma~\ref{lem:pants-transversality}, the multiplication map on
\([A]\times[B]\) is transverse to the conjugacy class \([AB]\).
Since conjugacy classes are the fibers of \(\vartheta\), the third
component
\[
(A',B')\longmapsto\vartheta((A'B')^{-1})
\]
has rank \(3\) on this kernel. Since the first two components are
already surjective onto a \(6\)-dimensional space, while the third is
surjective on their kernel, it follows that
\[
\operatorname{rank}d(\widetilde\beta_Y)_{(A,B)}=9.
\]

By Lemma~\ref{lem:proper-action}, the simultaneous conjugation action
is proper near \((A,B)\), and its stabiliser is finite by
\eqref{eq:H0-vanishes}. Choose a local slice through \((A,B)\).
The conjugation directions lie in the kernel of
\(d(\widetilde\beta_Y)_{(A,B)}\), so its restriction to the
\(15\)-dimensional slice still has rank \(9\). Passing to the quotient
by the finite stabiliser gives a local orbifold chart for
\(\mathcal X_Y\). Hence \(\beta_Y\) is a real-analytic submersion at
\([(A, B)]\), and its fiber has dimension \(15-9=6\).
\end{proof}
\begin{corollary}\label{cor:pants-local-coordinates}
On the intersection of the generic locus of
Lemma~\ref{lem:pants-boundary-submersion} with the dense open locus of
Theorem~\ref{thm:word-traces}, the relative deformation space
\[
\mathcal M_Y\bigl([A],[B],[(AB)^{-1}]\bigr)
\]
admits six real-analytic local coordinates. These coordinates may be taken
to be real-analytic functions of the fifteen trace coordinates of
Theorem~\ref{thm:word-traces}.
\end{corollary}

\begin{proof}
By Lemma~\ref{lem:pants-boundary-submersion}, the boundary map
\(
\beta_Y:\mathcal X_Y\longrightarrow\mathcal D^3
\)
is a real-analytic submersion of rank \(9\). Hence the real-analytic
submersion theorem gives local coordinates
\[
(u_1,\ldots,u_9,v_1,\ldots,v_6)
\]
on \(\mathcal X_Y\) such that
\[
\beta_Y=(u_1,\ldots,u_9).
\]
Thus \(v_1,\ldots,v_6\) restrict to local coordinates on each fiber of
\(\beta_Y\), in particular on
\(\mathcal M_Y([A],[B],[(AB)^{-1}])\).

On the dense open locus of Theorem~\ref{thm:word-traces}, the fifteen
trace functions form a real-analytic local coordinate system on
\(\mathcal X_Y\). Therefore the functions \(v_1,\ldots,v_6\) are locally
real-analytic functions of these fifteen trace coordinates.
\end{proof}
\subsection{Proof of Theorem~\ref{thm:FN-count}}\label{sec:FN-proof}

For each pair of pants \(Y_j\), write
\(A_j=\rho(\alpha_j)\), \(B_j=\rho(\beta_j)\), and
\(C_j=\rho(\delta_j)=(A_jB_j)^{-1}\). Each boundary circle of
\(Y_j\) comes from a cutting curve, so \(A_j,B_j,C_j\) are regular
loxodromic by the hypotheses of Theorem~\ref{thm:FN-count}.
Together with \eqref{eq:generic}, this shows that
Lemma~\ref{lem:pants-boundary-submersion} applies to every \(Y_j\).

Let \(\mathcal X_j\) be the corresponding local \(15\)-dimensional
orbifold and
\(\beta_j:\mathcal X_j\to\mathcal D^3\) its boundary-class map.
Each \(\beta_j\) is a submersion with \(6\)-dimensional fibers.
Therefore
\[
\beta=\prod_{j=1}^{2g-2}\beta_j:
\prod_{j=1}^{2g-2}\mathcal X_j
\longrightarrow\mathcal D^{6g-6}
\]
is a submersion near the point determined by \(\rho\).

Order the factors of \(\mathcal D^{6g-6}\) so that the two boundary
components arising from each cutting curve occur consecutively. Their
orientations are opposite, so the corresponding holonomies are inverse
up to conjugacy. Inversion acts on \(\mathcal D\) by
\[
\iota(x,y,z)=(y,x,z),
\]
because the first two coordinates are interchanged and
\(\nu(A^{-1})=\nu(A)\).

Let \(\Delta_{\mathcal C}\subset\mathcal D^{6g-6}\) be the product,
over the \(3g-3\) cutting curves, of the graphs of \(\iota\). Each
graph has codimension \(3\), so
\(\operatorname{codim}\Delta_{\mathcal C}=3(3g-3)\). Since
\(\beta\) is a submersion, the matching locus
\[
\mathcal P=\beta^{-1}(\Delta_{\mathcal C})
\]
is a real-analytic suborbifold of dimension
\[
15(2g-2)-3(3g-3)=21g-21.
\]

The parameters in \(\mathcal P\) have a simple interpretation. Each
cutting curve contributes one common point of \(\mathcal D\), hence
\(3(3g-3)\) parameters, while each of the \(2g-2\) pairs of pants
contributes a \(6\)-dimensional fiber after its three boundary classes
are fixed. Thus
\[
\dim\mathcal P
=
3(3g-3)+6(2g-2)
=
21g-21.
\]
These are the length  and internal-pants parameters.

It remains to glue the matched boundary components. Fix a cutting
curve \(\gamma_m\), choose representatives of the two adjacent
pair-of-pants classes, and let \(H_m^+\) and \(H_m^-\) be the
corresponding boundary holonomies with opposite orientations. Since
\([H_m^-]=[(H_m^+)^{-1}]\), there is \(g_m\in G\) such that
\[
g_mH_m^-g_m^{-1}=(H_m^+)^{-1}.
\]
Once one such \(g_m\) is fixed, every other choice is obtained by left
multiplication by an element of the centralizer of \(H_m^+\).
By Proposition~\ref{prop:centralizer}, this contributes \(3\)
parameters for each cutting curve, hence \(3(3g-3)\) twist parameters.
No quotient by the subgroup generated by the boundary holonomy is
taken: the surface is marked, and a full Dehn twist changes the marked
representation.

Thus the preceding argument identifies the three families of local
parameters. We now make a second count, before quotienting the
individual pair-of-pants representations, to verify that no further
positive-dimensional quotient remains.

Let \(\mathcal Z\) be the local space of tuples
\[
\bigl((\rho_j)_{j=1}^{2g-2},(g_m)_{m=1}^{3g-3}\bigr),
\]
where \(\rho_j\in\Hom(\pi_1(Y_j),G)\cong G^2\), the boundary
conjugacy classes match across the cuts, and the \(g_m\) satisfy the
corresponding gluing equations.

For each \(j\), let
\(\widetilde\beta_j:G^2\to\mathcal D^3\) denote the lifted boundary map
appearing in the proof of Lemma~\ref{lem:pants-boundary-submersion}. 
That proof shows that \(d\widetilde\beta_j\) has rank \(9\) near the
given representation. Hence the product of the
\(\widetilde\beta_j\) is a submersion near the point under
consideration. It follows that matching the boundary classes across
all cutting curves imposes exactly
\(3(3g-3)=9g-9\) conditions.

The raw pair-of-pants representations contribute
\(30(2g-2)=60g-60\) parameters. After the matching conditions are
imposed, the gluing elements contribute \(3(3g-3)=9g-9\) parameters.
The two contributions cancel in the count, and therefore
\[
\dim\mathcal Z=60g-60.
\]

There is a natural action of \(G^{2g-2}\) on \(\mathcal Z\). If
\(\gamma_m\) joins a boundary of \(Y_{j_-}\) to one of \(Y_{j_+}\),
then
\[
\rho_j\longmapsto h_j\rho_jh_j^{-1},
\qquad
g_m\longmapsto h_{j_+}g_mh_{j_-}^{-1}.
\]
The same formula applies when \(j_+=j_-\). This action records the
freedom in choosing representatives of the individual pair-of-pants
classes. By Lemma~\ref{lem:proper-action}, the action is proper
factor by factor, hence proper on \(\mathcal Z\); by
\eqref{eq:generic}, its stabilisers are finite. Its orbits therefore
have dimension \(15(2g-2)=30g-30\).

A point of \(\mathcal Z\) determines a nearby representation of
\(\pi_1(S)\): the representations on the individual pairs of pants are
combined using the gluing elements along the cutting curves. Replacing
the representative on \(Y_j\) by \(h_j\rho_jh_j^{-1}\) changes the
adjacent gluing elements by the formula above, but leaves the resulting
global conjugacy class unchanged.

Conversely, suppose two nearby points of \(\mathcal Z\) determine the
same global conjugacy class. After conjugating one of the assembled
representations, we may assume that the two global representations agree.
For each pair of pants \(Y_j\), the two local representatives are then
two choices of representative for the restriction of the same global
representation to \(Y_j\). Hence they differ by a conjugation
\(h_j\in G\):
\[
\rho'_j=h_j\rho_jh_j^{-1}.
\]
The gluing elements record the identifications of these local
representatives along the cutting curves. Therefore compatibility with
the common global representation gives, for every cutting curve
\(\gamma_m\),
\[
g'_m=h_{j_+}g_mh_{j_-}^{-1}.
\]
Thus the two points lie in the same \(G^{2g-2}\)-orbit. Since the
\(G^{2g-2}\)-action leaves the assembled global conjugacy class unchanged,
its orbits are exactly the fibers of the map from \(\mathcal Z\) to global
conjugacy classes.

Therefore \(\mathcal Z/G^{2g-2}\) is locally the deformation space near
\([\rho]\), and
\[
\dim\mathcal Z-15(2g-2)
=
60g-60-(30g-30)
=
30g-30.
\]
Equivalently,
\[
30g-30
=
\underbrace{3(3g-3)}_{\text{lengths}}
+
\underbrace{6(2g-2)}_{\text{internal pants}}
+
\underbrace{3(3g-3)}_{\text{twists}}.
\]
This proves Theorem~\ref{thm:FN-count}.
\qed

\section{Toward a global Fenchel--Nielsen coordinate system}
\label{sec:open}

The results of the preceding sections provide the local structure expected
from a Fenchel--Nielsen theory. What remains is to globalise these parameters.

\subsection{Pants parameters}

The first issue is to find six geometrically meaningful global parameters
for the relative representation space of a pair of pants with prescribed
regular loxodromic boundary conjugacy classes. The local trace coordinates
of Theorem~\ref{thm:word-traces} and the axial description of
Section~\ref{sec:frames} provide two possible starting points.

Such parameters should determine the relative conjugacy class and behave
naturally under the cyclic permutation
\[
A\longmapsto B\longmapsto (AB)^{-1}\longmapsto A.
\]
One must also determine which triples of boundary conjugacy classes in
$\mathcal D^3$ are realised by a pair of pants. This would give the analogue
of the classical Fricke domain.

\subsection{Twist parameters}

The local twist parameters arise from the centralizer of a regular
loxodromic boundary element. To obtain global twist coordinates, one must
choose global parameters on this centralizer, taking account of the
periodicity of the rotational factors, and describe explicitly the action
of a full Dehn twist. This should also clarify the induced mapping class
group action on the coordinates.

\subsection{Global parametrisation}

For a fixed pants decomposition, the remaining step is to identify the
global parameter domain and prove a reconstruction and uniqueness theorem:
every admissible collection of boundary, internal pants, and twist parameters
should determine a representation, uniquely up to conjugacy on the chosen
domain.

One must also understand the compatibility of these parameters under changes
of pants decomposition. A natural setting in which to begin is a suitable
open subset or component containing the Fuchsian locus
\[
\SL(2,\R)\hookrightarrow\SL(2,\H).
\]
It would then be important to describe the discrete faithful and convex
cocompact representations in these coordinates.

\appendix

\section{Exact verification of the Jacobian computation}
\label{app:jacobian-computation}

We record here a SageMath computation verifying the Jacobian minor used
in the proof of Theorem~\ref{thm:word-traces}. The computation is
carried out over \(\Q(i)\), using the complex embedding
\(\Phi:M(2,\H)\to M(4,\C)\) from \eqref{eq:Phi-quat} and the identity
\[
\tr C=\frac12\Tr_{\C}\Phi(C)
\]
from Lemma~\ref{lem:trace-is-complex-trace}.

The SageMath code used for this verification was generated with the
assistance of ChatGPT (OpenAI), and was subsequently checked and
executed by the authors using exact arithmetic.

At the test pair \eqref{eq:test-pair}, let
\[
J=d\widetilde{\mathcal T}_{(A,B)}.
\]
The basis of \(\mathfrak g\) and the ordering of the thirty tangent
directions are exactly those specified in the proof of
Theorem~\ref{thm:word-traces}: columns \(1,\ldots,15\) are the
\(X\)-directions and columns \(16,\ldots,30\) are the \(Y\)-directions.

The following SageMath code constructs the test pair, the ordered basis
of \(\mathfrak g\), the fifteen words in \(\mathcal W\), and the
Jacobian matrix \(J\).

\begin{verbatim}
K.<ii> = QuadraticField(-1)

Z2 = zero_matrix(K, 2)
I2 = identity_matrix(K, 2)

qi = matrix(K, [[ii, 0],
                [0, -ii]])

qj = matrix(K, [[0, 1],
                [-1, 0]])

qk = matrix(K, [[0, ii],
                [ii, 0]])

def Qmat(a, b, c, d):
    return block_matrix([[a, b],
                         [c, d]])

A = Qmat(2*qj, Z2, Z2, (1/2)*qi)
D = Qmat(-2*qk, Z2, Z2, (1/2)*qj)
P = Qmat(I2, qi, I2, I2)

B = P * D * P.inverse()

Ai = A.inverse()
Bi = B.inverse()

basis = []

# diag(i,0), diag(j,0), diag(k,0)
for u in [qi, qj, qk]:
    basis.append(Qmat(u, Z2, Z2, Z2))

# diag(0,i), diag(0,j), diag(0,k)
for u in [qi, qj, qk]:
    basis.append(Qmat(Z2, Z2, Z2, u))

# diag(1,-1)
basis.append(Qmat(I2, Z2, Z2, -I2))

# E_12 u, u = 1,i,j,k
for u in [I2, qi, qj, qk]:
    basis.append(Qmat(Z2, u, Z2, Z2))

# E_21 u, u = 1,i,j,k
for u in [I2, qi, qj, qk]:
    basis.append(Qmat(Z2, Z2, u, Z2))

letters = {
    'A': A,
    'a': Ai,
    'B': B,
    'b': Bi
}

words = [
    ['A'],
    ['a'],
    ['A','A'],
    ['B'],
    ['b'],
    ['B','B'],
    ['A','B'],
    ['A','b'],
    ['a','B'],
    ['a','b'],
    ['A','A','B'],
    ['A','B','B'],
    ['A','A','B','B'],
    ['A','B','A','b'],
    ['A','B','a','b']
]

def rtr(C):
    return C.trace()/2

def derivative(word, X=None, Y=None):
    if X is None:
        X = zero_matrix(K, 4)
    if Y is None:
        Y = zero_matrix(K, 4)

    dletter = {
        'A': X*A,
        'a': -Ai*X,
        'B': Y*B,
        'b': -Bi*Y
    }

    total = zero_matrix(K, 4)

    for k in range(len(word)):
        term = identity_matrix(K, 4)
        for j in range(len(word)):
            L = word[j]
            if j == k:
                term *= dletter[L]
            else:
                term *= letters[L]
        total += term

    return rtr(total)

rows = []

for W in words:
    row = []

    for X in basis:
        row.append(derivative(W, X=X))

    for Y in basis:
        row.append(derivative(W, Y=Y))

    rows.append(row)

J = matrix(K, rows)

# The entries are in fact rational.
assert all(x in QQ for x in J.list())
J = matrix(QQ, J)

M = 32*J

assert all(x in ZZ for x in M.list())

# Sage uses zero-based indexing.
cols = list(range(12)) + [15,16,25]

Jminor = J.matrix_from_columns(cols)
Mminor = M.matrix_from_columns(cols)

print("rank J =", J.rank())
print("max |entry of 32J| =", max(abs(x) for x in M.list()))
print("det(32J minor) =", factor(Mminor.det()))
print("det(J minor) =", factor(Jminor.det()))
\end{verbatim}

The output is

\begin{verbatim}
rank J = 15
max |entry of 32J| = 800
det(32J minor) = 2^41 * 3^12 * 5^11 * 191
det(J minor) = 3^12 * 5^11 * 191 / 2^34
\end{verbatim}

Thus \(32J\) has integer entries, and the \(15\times15\) submatrix
determined by columns \(1,\ldots,12,16,17,26\) satisfies
\[
\det (32J)_{\,1,\ldots,12,16,17,26}
=
2^{41}\cdot3^{12}\cdot5^{11}\cdot191.
\]
Since \(32^{15}=2^{75}\), the corresponding minor of \(J\) is
\[
\det J_{\,1,\ldots,12,16,17,26}
=
\frac{3^{12}\cdot5^{11}\cdot191}{2^{34}}
\neq0.
\]
Since \(J=d\widetilde{\mathcal T}_{(A,B)}\) is a \(15\times30\)
matrix, its rank is at most \(15\). The nonvanishing of the above
\(15\times15\) minor shows that the corresponding fifteen columns are
linearly independent.  Therefore, 
\(
\operatorname{rank}d\widetilde{\mathcal T}_{(A,B)}=15.
\)

\subsection*{Acknowledgements}
The authors thank John Parker for valuable discussions over the years, including discussions on this problem during the ICTS program ``New Trends in Teichmüller Theory'' (ICTS/nteich2025/02).

Parts of this work were carried out during visits by Gongopadhyay to The Institute of Mathematical Sciences, Chennai, in August 2025 and subsequently to the Institut des Hautes Études Scientifiques (IHES) in 2025--26; by Kalane to the Institut Henri Poincaré (IHP) for the program ``Higher Rank Geometric Structures'' in April--July 2025 and subsequently as a postdoctoral fellow at The Institute of Mathematical Sciences, Chennai, during 2025--26; and by Mukherjee to IISER Mohali in November 2023 and April 2024. The authors thank these institutions for their hospitality and support.

Gongopadhyay acknowledges support from the ANRF project   ARGM/2025/000122/MTR, and Mukherjee acknowledges support from an NBHM research grant.

\section*{Declarations} 

{\bf Ethical Approval}. Not applicable.

{\bf Competing Interest}. Not applicable.

{\bf Funding}. Not applicable.

{\bf Authors' Contributions.} All authors have contributed equally. 

{\bf Availability of data and materials}. Not applicable.

\medskip \textbf{AI declaration.} The first draft of the paper was completed without AI assistance. Thereafter, the authors used ChatGPT (OpenAI) and Claude (Anthropic) in a limited capacity to verify proofs and clarify some of the  mathematical arguments. ChatGPT was also used to assist with generating the SageMath code for the symbolic verification in Appendix~\ref{app:jacobian-computation}. All AI-assisted output was independently reviewed and verified by the authors, who retain full responsibility for the content of the manuscript.




\begin{thebibliography}{99}





\bibitem{cpw} W. Cao, J. R. Parker and X. Wang, \emph{On the classification of
quaternionic M\"obius transformations}, Math. Proc. Cambridge Philos. Soc.
\textbf{137} (2004), no. 2, 349--361.


\bibitem{cg} S. S. Chen and L. Greenberg, \emph{Hyperbolic spaces}, in: Contributions to
Analysis, Academic Press, New York, 1974, 49--87.


\bibitem{fp} R. D\'avila~Figueroa and J.~R. Parker, \emph{Fenchel-Nielsen coordinates for ${\rm SL}(3,\mathbb{C})$ representations}, Geom. Dedicata { 219} (2025), no.~4, Paper No. 63, 31 pp.; MR4921558. 

\bibitem{DVV}
J.~P. D\'iaz, A.~Verjovsky and F.~Vlacci,
\emph{Quaternionic Kleinian modular groups and arithmetic hyperbolic
orbifolds over the quaternions},
Geom. Dedicata {192} (2018), 127--155.
\bibitem{DS} D. ~\v{Z}.~\DJ{}okovi\'c and B.~H. Smith, \emph{Quaternionic matrices: unitary similarity, simultaneous triangularization and some trace identities}, Linear Algebra Appl. {\ 428} (2008), no.~4, 890--910; MR2382098.






\bibitem{foreman} B. Foreman, \emph{Conjugacy invariants of $\SL(2,\H)$},
Linear Algebra Appl. {381} (2004), 25--35.



\bibitem{goldman} W. M. Goldman, \emph{The symplectic nature of fundamental groups of surfaces},
Adv. Math. {54} (1984), no. 2, 200--225.

\bibitem{gold1} W. M. Goldman, \emph{Convex real projective structures on compact
surfaces}, J. Differential Geom. {31} (1990), 791--845.

\bibitem{gold2} W. M. Goldman, \emph{Trace coordinates on Fricke spaces of some simple
hyperbolic surfaces}, in: Handbook of Teichm\"uller Theory, Vol. II, IRMA Lect. Math.
Theor. Phys. 13, Eur. Math. Soc., Z\"urich, 2009, 611--684.


\bibitem{kg} K. Gongopadhyay, \emph{Algebraic characterization of the isometries of the
hyperbolic $5$-space}, Geom. Dedicata {144} (2010), 157--170.

\bibitem{kgsk} K. Gongopadhyay and S. B. Kalane, \emph{Quaternionic hyperbolic
Fenchel--Nielsen coordinates}, Geom. Dedicata {199} (2019), 247--271.
\bibitem{kgsk2} K. Gongopadhyay and S.~B. Kalane, \emph{Local coordinates for complex and quaternionic hyperbolic pairs}, J. Aust. Math. Soc. {113} (2022), no.~1, 57--78; MR4450912.

\bibitem{gkl} K. Gongopadhyay and R. S. Kulkarni, \emph{$z$-classes of isometries of the
hyperbolic space}, Conform. Geom. Dyn. {13} (2009), 91--109.

\bibitem{gl-lawton}K. Gongopadhyay and S. Lawton, \emph{Invariants of pairs in ${\rm SL}(4, \mathbb{C})$ and ${\rm SU}(3, 1)$}, Proc. Amer. Math. Soc. {145} (2017), no.~11, 4703--4715; MR3691988. 




\bibitem{gl} E. Gwynne and M. Libine;  \emph{On a quaternionic analogue of the cross ratio}. Adv. Appl. Clifford Algebras 22(4), 1041–1053 (2012)
\bibitem{kour} C. Kourouniotis, \emph{Complex length coordinates for quasi-Fuchsian
groups}, Mathematika 41 (1994), 173--188.
\bibitem{ki1}
Y.~Kim,
\emph{Quasiconformal stability for isometry groups in hyperbolic \(4\)-space},
Bull. Lond. Math. Soc. {43} (2011), no.~1, 175--187.

\bibitem{ki2} Y. Kim, \emph{Thrice-punctured sphere groups in hyperbolic $4$-space}, Proc. Amer. Math. Soc. { 151} (2023), no.~6, 2679--2693; MR4576329. 
\bibitem{LeeSmoothManifolds}
J.~M.~Lee,
\emph{Introduction to Smooth Manifolds},
2nd ed.,
Graduate Texts in Mathematics, vol.~218,
Springer, New York, 2013.
\bibitem{Mas} B. Maskit, \emph{Matrices for Fenchel--Nielsen coordinates},
Ann. Acad. Sci. Fenn. Math. {26} (2001), 267--304.
\bibitem{park1} J. R. Parker, \emph{Traces in complex hyperbolic geometry}, in:
Geometry, Topology and Dynamics of Character Varieties, Lect. Notes Ser. Inst. Math.
Sci. Natl. Univ. Singap. 23, World Scientific, 2012, 191--245.
 \bibitem{pa2} J. R. Parker, \emph{Hyperbolic Spaces}, Jyv\"askyl\"a Lectures in
Mathematics 2, University of Jyv\"askyl\"a, 2008.

\bibitem{parkplat} J. R. Parker and I. D. Platis, \emph{Complex hyperbolic
Fenchel--Nielsen coordinates}, Topology {47} (2008), 101--135.

\bibitem{ps} J. R. Parker and I. Short, \emph{Conjugacy classification of quaternionic
M\"obius transformations}, Comput. Methods Funct. Theory {9} (2009), no. 1,
13--25.



\bibitem{tan} S. P. Tan, \emph{Complex Fenchel--Nielsen coordinates for quasi-Fuchsian
structures}, Internat. J. Math. {5} (1994), 239--251.




\bibitem{wat} P. Waterman, \emph{M\"obius groups in several dimensions}, Adv. Math.
{101} (1993), 87--113.

\bibitem{wi} J. B. Wilker, \emph{The quaternion formalism for M\"obius groups in four or
fewer dimensions}, Linear Algebra Appl. {190} (1993), 99--136.

\bibitem{wil} P. Will, \emph{Traces, cross-ratios and $2$-generator subgroups of
$\SU(2,1)$}, Canad. J. Math. {61} (2009), no. 6, 1407--1436.



\end{thebibliography}
\end{document}